\documentclass{imammb}

\jno{dqnxxx}

\usepackage{natbib}

\usepackage{graphicx}
\usepackage{amsmath,amsthm,bm,mathrsfs}
\usepackage{color}

\renewenvironment{proof}[1][Proof]{\noindent\textit{#1. } }{\hfill$\square$}

 \newtheoremstyle{theorem}{6pt}{6pt}{\rm}{}{\sffamily}{ }{ }{}
 \theoremstyle{theorem}
\newtheorem{theorem}{\sc Theorem}[section]

 \newtheoremstyle{algorithm}{6pt}{6pt}{\rm}{}{\sffamily}{ }{ }{}
 \theoremstyle{algorithm}

 \newtheoremstyle{lemma}{6pt}{6pt}{\rm}{}{\sffamily}{ }{ }{}
 \theoremstyle{lemma}
 \newtheorem{lemma}{\sc Lemma}[section]

\newtheoremstyle{case}{6pt}{6pt}{\rm}{}{\sffamily}{. }{ }{}
 \theoremstyle{case}

 \newtheoremstyle{statement}{6pt}{6pt}{\rm}{}{\sffamily}{ }{ }{}
\theoremstyle{statement}

 \newtheoremstyle{corollary}{6pt}{6pt}{\rm}{}{\sffamily}{ }{ }{}
 \theoremstyle{corollary}
 \newtheorem{corollary}{\sc Corollary}[section]
 
  \newtheoremstyle{definition}{6pt}{6pt}{\rm}{}{\sffamily}{ }{ }{}
 \theoremstyle{definition}
 \newtheorem{definition}{\sc Definition}[section]

\newtheoremstyle{example}{6pt}{6pt}{\rm}{}{\sffamily}{ }{ }{}
\theoremstyle{example}

\newtheoremstyle{remark}{6pt}{6pt}{\rm}{}{\sffamily}{ }{ }{}
\theoremstyle{remark}
\newtheorem{remark}{\sc Remark}[section]

\newtheoremstyle{approximation}{6pt}{6pt}{\rm}{}{\sffamily}{ }{ }{}
\theoremstyle{approximation}

\newtheoremstyle{scheme}{6pt}{6pt}{\rm}{}{\sffamily}{ }{ }{}
\theoremstyle{scheme}

\newtheoremstyle{Algorithm}{6pt}{6pt}{\rm}{}{\sffamily}{ }{ }{}
\theoremstyle{Algorithm}

\newtheoremstyle{Assumption}{6pt}{6pt}{\rm}{}{\sffamily}{ }{ }{}
\theoremstyle{Assumption}

\newtheoremstyle{proposition}{6pt}{6pt}{\rm}{}{\sffamily}{ }{ }{}
\theoremstyle{proposition}
\newtheorem{proposition}{\sc Proposition}[section]

\newtheoremstyle{hypo}{6pt}{6pt}{\rm}{}{\sffamily}{ }{ }{}
 \theoremstyle{hypo}

  \newtheoremstyle{Step}{6pt}{6pt}{\rm}{}{}{ }{ }{}
 \theoremstyle{Step}

\numberwithin{equation}{section}

\newcommand{\df}{\displaystyle\frac}
\newcommand{\la}{\lambda}

\newcommand{\om}{\omega}

\begin{document}

\title{Stability and Hopf bifurcation analysis for the hypothalamic-pituitary-adrenal axis model with memory}

\author{{\sc Eva Kaslik$^{1,2}$, Mihaela Neam\c{t}u$^{1,3}$}\\[2pt]
$^1$West University of Timi\c{s}oara, Bd. V. Parvan nr. 4, 300223, Romania, \\
$^2$Institute e-Austria Timisoara, Bd. V. Parvan nr. 4, cam. 045B, 300223, Romania\\
$^3$University Politehnica of Bucharest, 313 Splaiul Independentei, 060042 Bucharest, Romania\\[6pt]
{\rm [Received on April 5, 2016. Revised on August 11, 2016]}\vspace*{6pt}}

\pagestyle{headings}
\markboth{E. KASLIK, M. NEAM\c{T}U}{\rm HYPOTHALAMIC-PITUITARY-ADRENAL AXIS MODEL WITH MEMORY}
\maketitle


\begin{abstract}
{This paper generalizes the existing minimal model of the hypothalamic-pituitary-adrenal (HPA) axis in a realistic way, by including memory terms: distributed time delays, on one hand and frac\-tional-order derivatives, on the other hand.  The existence of a unique equilibrium point of the mathematical models is proved and a local stability analysis is undertaken for the system with general distributed delays. A thorough bifurcation analysis for the distributed delay model with several types of delay kernels is provided. Numerical simulations are carried out for the distributed delays models and for the fractional-order model with discrete delays, which substantiate the theoretical findings. It is shown that these models are able to capture the vital mechanisms of the HPA system.}

{HPA axis; stability; Hopf bifurcation; distributed delays; fractional-order derivatives.}
\end{abstract}

\section{Introduction}

The hypothalamus–pituitary–adrenal (HPA) axis is a self-regulated dynamic feedback neuroendocrine system that is engaged in the rapid response to stressful stimuli and is responsible for the return to homeostasis through complex feedback mechanisms \cite{Conrad_2009,Swanson_2000}. By regulating the plasma levels of corticosteroids secreted from adrenal glands, it also controls many bodily processes, including mood and emotions, digestion, sexuality, the immune system, energy storage and expenditure.

The HPA axis is organized into three distinct regions: the hypothalamus, pituitary gland and adrenal gland, with a complex set of direct influences and feedback interactions among the three endocrine glands.  These glands work together by producing and secreting, or responding to common hormones including corticotropin-releasing hormone (CRH), corticotropin (ACTH), and cortisol (CORT) \cite{Kyrylov_2005}.

Both physical stressors (e.g. infection, thermal exposure, dehydration) and psychological stressors (e.g. fear, anticipation) activate the hypothalamus to release CRH, which induces the ACTH production in the pituitary. Then, ACTH is transported by the blood to the adrenal cortex of the adrenal gland, where it stimulates the production of cortisol, which in turn suppresses the production of both CRH and ACTH (see Fig. \ref{fig:0}).

\begin{figure}
\centering
\includegraphics*[width=0.3\textwidth]{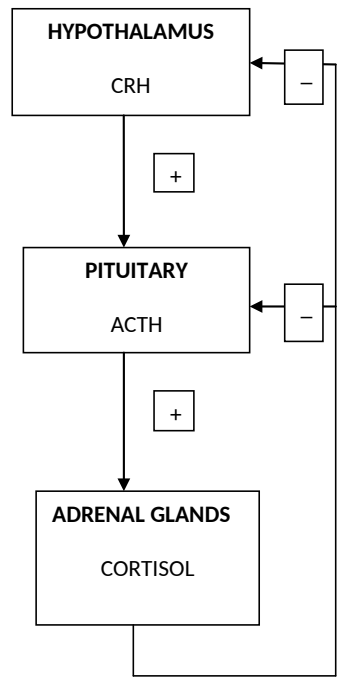}
\caption{A simple schematic representation of the HPA axis with the negative feedback.}
\label{fig:0}
\end{figure}

It is important to keep cortisol concentration within a certain physiological range. Cortisol overproduction,
which is often due to a pituitary tumour causing high levels of circulating ACTH, leads to Cushing’s disease.
Cortisol underproduction, which generates Addison’s disease, is most frequently the result of adrenal destruction.

In the past few decades, mathematical modelling has started to play an increasingly important role in the study of metabolic and endocrine processes, both in physiology and in clinical medicine \cite{Bairagi_2008}. The theory of nonlinear dynamical systems has become a promising research tool for studying rhythmicity in humans. Moreover, it is often mandatory to introduce time delays in the mathematical models describing real life phenomena.

Analyzing a mathematical model of the HPA axis is helpful in understanding the simultaneous feedback mechanisms and revealing the different ways in which a malfunction can occur \cite{Andersen_2013}. It is important to emphasize that time delays { of up to 60 min (according to \cite{Boscaro_1998,Posener_1997}) } unavoidably exist in the HPA axis, due to the transportation of the hormones among the three glands.

Several mathematical models of the HPA axis have been developed and analyzed by \cite{Andersen_2013,Bairagi_2008,Conrad_2009,Gudmand_2014,Jelic_2005,Kyrylov_2005,Lenbury_2005,Markovic_2011,Pornsawad_2013,Savic_2005,Savic_2006,Vinther_2011}.
Experimental data show the circadian as well as ultradian rhythm of hormone levels { \cite{Carroll_2007,Veldhuis_2008} }, which should be reflected by the mathematical models of the HPA axis, through the existence of oscillatory solutions. The ultradian rhythm is usually considered an inherent behavior of the HPA axis, whereas the circadian rhythm is regarded as an external input to the axis \cite{Andersen_2013}.

A frequently studied model of the HPA axis, called "minimal model" \cite{Vinther_2011}, consists of a system of three coupled, non-linear differential equations, with the hormones CRH, ACTH and { free } cortisol as variables. { While the mathematical model investigated by \cite{Kyrylov_2005} included two more differential equations, accounting for corticosteroid-binding globulin (CBG) bound cortisol, and albumin bound cortisol, \cite{Vinther_2011} pointed out that only free cortisol is capable of interacting with the rest of the HPA-axis, which is the reason for considering only three differential equations in the formulation of the minimal model. } If time delays are not incorporated in the mathematical model, no oscillatory behaviour has been detected \cite{Andersen_2013,Savic_2005,Vinther_2011}. For systems of ordinary differential equations, sufficiently large time delays are often linked with generating oscillatory solutions.

A particular case of the minimal model with discrete time-delays and exponential negative-feedback mechanism has been studied in \cite{Lenbury_2005} and later, in \cite{Pornsawad_2013}, revealing the occurrence of Hopf bifurcations resulting in the emergence of periodic orbits. Other variants of the minimal model with Hill-type feedback functions and discrete time-delays have been considered in \cite{Savic_2005,Savic_2006}, but no oscillations have been reported. More recently, Vinther et al. \cite{Vinther_2011} have included discrete time-delays in the minimal model and have observed periodic solutions by computer simulations. Their investigations show that time delays of at least 18 min in the feedback mechanisms are needed for generating oscillations. A slightly modified version of the minimal model with discrete time-delays has been analyzed in \cite{Bairagi_2008}, successfully obtaining the desired Hopf bifurcation and thereby, oscillating solutions for sufficiently large time-delays.


In this paper, we generalize the existing minimal model of the HPA axis in two realistic ways: firstly, by including distributed time delays and secondly, by considering frac\-tional-order derivatives.

On one hand, distributed time delays represent the situation where the delays occur in certain ranges of values with some associated probability distributions, taking into account the whole past history of the variables. In many real world applications, distributed time delays are more realistic and more accurate than discrete time delays \cite{Cushing_2013}. Distributed delay models appear in a wide range of applications such as, population biology \cite{Faria_2008,Ruan_1996}, hematopoiesis \cite{Adimy_2003,Adimy_2005,Adimy_2006,Ozbay_2008}, neural networks \cite{Jessop_2010}.

On the other hand, the main benefit of fractional-order models in comparison with classical integer-order models is that fractional derivatives provide a good tool for the description of memory and hereditary properties of various processes \cite{Kilbas,Lak,Podlubny}.

{ Due to the fact that the whole past history of the variables is accounted for in the formulation of both distributed time-delays as well as fractional-order derivatives, these generalizations of the minimal model are able to reflect biological variability in a better way than other approaches. }

The paper is structured as follows. Section 2 provides the mathematical model of the HPA axis, where instead of considering the transportation of different hormones as an instantaneous process, we introduce distributed time delays to account for the time needed by the hormones to travel from source to destination. In Section 3, the existence of a unique equilibrium point of the system is shown. Local stability analysis of the system with general distributed delays is analyzed in Section 4. In Section 5, we undertake a bifurcation analysis for the distributed delay model in the case of several types of delay kernels. The fractional order mathematical model of the HPA axis is presented and shortly analyzed in Section 6. Numerical simulations are carried out and discussed in Section 7, followed by concluding remarks in Section 8.

\section{Mathematical model of HPA with distributed delays}

In formulating the mathematical model which describes the variation in time of the concentrations of the three hormones CRH, ACTH and CORT, the following sequence of typical events is considered, according to the schema presented in Fig. \ref{fig:0}. CRH is secreted from the hypothalamus and released into the portal blood vessel of the hypophyseal stalk, and then transported to the anterior pituitary where it stimulates the secretion of ACTH, with an average time delay $\tau_1$. Then, in the cortex of the adrenal glands, ACTH stimulates the secretion of the stress hormone cortisol with the average time delay $\tau_2$. Cortisol has a negative feedback effect on the hypothalamus and the pituitary, expressed by two feedback functions $f_1$ and $f_2$, affecting the synthesis and release of CRH and ACTH, respectively. On one hand, cortisol inhibits the secretion of CRH through glucocorticoid receptors (GRs) situated in the hypothalamus \cite{Landsberg_1992}, with an average time delay $\tau_{31}$. On the other hand, cortisol also performs a negative feedback on the secretion of ACTH through GRs situated in the pituitary, with an average time delay $\tau_{32}$. The hormone concentrations of CRH, ACTH and cortisol are depleted through the rate constants $w_1$, $w_2$ and $w_3$, respectively.

The mathematical model of the HPA axis studied in this paper is based on the minimal model introduced in \cite{Vinther_2011}. The main improvement is that we consider distributed time delays to account for the transport of the hormones among the glands. Denoting the hormone concentrations, for simplicity, by $CRH(t)=x_1(t)$, $ACTH(t)=x_2(t)$, $CORT(t)=x_3(t)$, the following system of differential equations with distributed delays is considered:
\begin{equation}\label{sys.hpa.dd}
\left\{\begin{array}{l}
\dot x_1(t)=f_1\left(\int_{-\infty}^t x_3(s)h_{31}(t-s)ds\right)-w_1x_1(t),\\
\dot x_2(t)=f_2\left(\int_{-\infty}^t x_3(s)h_{32}(t-s)ds\right)\int_{-\infty}^tx_1(s)h_1(t-s)ds-w_2x_2(t),\\
\dot  x_3(t)=k_3\int_{-\infty}^tx_2(s)h_2(t-s)ds-w_3x_3(t),
\end{array}\right.
\end{equation}
where all the first terms on the right hand side represent production and all the second terms represent depletion of hormones. The constant $k_3$ as well as the elimination constants $w_1,w_2,w_3$ are positive.

The functions $f_1,f_2:[0,\infty)\to (0,\infty)$, which represent the negative feedback from CORT on CRH and ACTH, respectively, are assumed to be strictly decreasing, smooth and bounded on $[0,\infty)$. In particular, the results presented in this paper are also applicable when Hill functions are being used in the expression of the feedback functions \cite{Andersen_2013,Vinther_2011}:
\begin{equation}\label{func.hill}
f_1(u)=k_1\left(1-\eta\df{u^{\alpha}}{c^{\alpha}+u^{\alpha}}\right)\quad,\quad f_2(u)=k_2\left(1-\mu\df{u^{\alpha}}{c^{\alpha}+u^{\alpha}}\right),
\end{equation}
with $\alpha\geq 1$, $k_1,k_2>0$, $\eta,\mu\in(0,1)$, $c>0$. It is easy to verify that functions (\ref{func.hill}) satisfy all the properties mentioned above. However, it may be possible to model the negative feedback using different types of functions $f_1$ and $f_2$. In this paper, our aim is to obtain general results which will also be applicable to other choices of negative feedback functions, besides functions (\ref{func.hill}), often used in the literature.

In system (\ref{sys.hpa.dd}), the delay kernels $h_1,h_2,h_{31},h_{32}:[0,\infty)\to[0,\infty)$ are probability density functions representing the probability that a particular time delay occurs. They are assumed to be bounded, piecewise continuous and satisfy
\begin{equation}\label{delay.kernel.properties}
\int_0^{\infty}h(s)ds=1.
\end{equation}
The average delay of a delay kernel $h(t)$ is given by
$$\tau=\int_0^{\infty}sh(s)ds<\infty.$$
Two important classes of delay kernels, which are often used in the literature, are worth mentioning:
\begin{itemize}
\item Dirac kernels: $h(s)=\delta(s-\tau)$, where $\tau\geq 0$. In this particular case, the distributed delay is reduced to a discrete time delay:
$$\int_{-\infty}^t x(s)h(t-s)ds=\int_0^\infty x(t-s)\delta(s-\tau)ds=x(t-\tau).$$
\item Gamma kernels: $h(s)=\df{ s^{p-1}e^{-s/\beta}}{\beta^p\Gamma(p)}$, where $p,\beta>0$. The average delay of a Gamma kernel is $\tau=p\beta$.
\end{itemize}
The analysis of the mathematical models including particular classes of delay kernels (such as weak Gamma kernels with $p=1$ or strong Gamma kernels with $p=2$) may shed a light on how distributed delays affect the dynamics differently from discrete delays. However, in the modelling of real world phenomena, one usually does not have access to the exact distribution, and approaches using general kernels may be more useful \cite{Bernard_2001,Campbell_2009,Diekmann_2012,Yuan_2011}.

Initial conditions associated with system (\ref{sys.hpa.dd}) are as follows:
$$x_i(s)=\varphi_i(s),\quad \forall\, s\in(-\infty,0], \,\,i=1,2,3,$$
where $\varphi_i$  are bounded continuous functions defined on $(-\infty,0]$, with values in $[0,\infty)$.

\section{Existence of a unique equilibrium point}

An equilibrium point of system (\ref{sys.hpa.dd}) is a solution of the following algebraic system:
\begin{equation}
\left\{\begin{array}{l}
f_1(x_3)=w_1x_1,\\
f_2(x_3)x_1=w_2x_2,\\
k_3x_2=w_3x_3,
\end{array}\right.
\end{equation}
which is equivalent to
\begin{equation}\label{sys.eq.points}
\left\{\begin{array}{l}
x_1=\df{f_1(x_3)}{w_1},\\
x_2=\df{w_3}{k_3}x_3,\\
x_3=\df{k_3}{w_1w_2w_3}f_1(x_3)f_2(x_3).
\end{array}\right.
\end{equation}
Due to the properties of $f_1$ and $f_2$, the function $\df{k_3}{w_1w_2w_3}f_1(x)f_2(x)$ which appears in the right hand side of the last equation of system (\ref{sys.eq.points}), is strictly  positive and  strictly decreasing on $[0,\infty)$, and therefore, it has a unique fixed point $x^\star>0$. It follows that system (\ref{sys.hpa.dd}) has a unique equilibrium point
\begin{equation}\label{eq.state}
E=\left(\df{f_1(x^\star)}{w_1}, \df{w_3x^\star}{k_3},x^\star\right).
\end{equation}

In the following, necessary and sufficient conditions will be explored for the local asymptotic stability of the equilibrium point $E$ and the occurrence of limit cycles in a neighborhood of $E$ (due to Hopf bifurcations) that can explain the ultradian rhythm.

\section{Local stability analysis of system (\ref{sys.hpa.dd})}

In this section, considering general delay kernels, we seek to obtain delay independent
sufficient conditions for the local asymptotic stability of the equilibrium point $E$. Such results prove to be useful if one is unable to accurately estimate the time delays in system (\ref{sys.hpa.dd}).

{ Using the transformation $y_1(t)=x_1(t)-\df{f_1(x^\star)}{w_1}$, $y_2(t)=x_2(t)-\df{w_3x^\star}{k_3}$ and $y_3(t)=x_3(t)-x^\star$, the linearized system of (\ref{sys.hpa.dd}) at the equilibrium point $E$ is}:
\begin{equation}\label{sys.lin}
\left\{\begin{array}{l}
\dot y_1(t)=f_1'(x^\star)\int_{-\infty}^ty_3(s)h_{31}(t-s)ds-w_1y_1(t),\\
\dot y_2(t)=f_2(x^\star)\int_{-\infty}^ty_1(s)h_1(t-s)ds+\df{1}{w_1}f_1(x^\star)f_2'(x^\star)\int_{-\infty}^ty_3(s)h_{32}(t-s)ds-w_2y_2(t),\\
\dot y_3(t)=k_3\int_{-\infty}^ty_2(s)h_2(t-s)ds-w_3y_3(t).
\end{array}\right.
\end{equation}

The associated characteristic equation of the linearized system (\ref{sys.lin}) is:
\begin{equation}\label{eq.char}
( z+w_1)( z+w_2)( z+w_3)+a( z+w_1)H_2( z)H_{32}( z)+bH_1( z)H_2( z)H_{31}( z)=0,
\end{equation}
where $H_i( z)=\int_{0}^\infty e^{- z s}h_i(s)ds$ represent the Laplace transforms of the delay kernels $h_i$, $i\in\{1,2,31,32\}$ and
\begin{align}
\label{eq.a} a&=-\df{k_3}{w_1}f_1(x^\star)f_2'(x^\star)=-w_2w_3\df{x^\star f_2'(x^\star)}{f_2(x^\star)}>0,\\
\label{eq.b} b&=-k_3f_1'(x^\star)f_2(x^\star)=-w_1w_2w_3\df{x^\star f_1'(x^\star)}{f_1(x^\star)}>0.
\end{align}

The following inequalities will be useful for the theoretical analysis:
\begin{align*}
(I_1)\qquad & 8f_1(x^\star)+x^\star f_1'(x^\star)\geq 0;\\
(I_2)\qquad & 1+\frac{x^\star f_1'(x^\star)}{f_1(x^\star)}+\frac{x^\star f_2'(x^\star)}{f_2(x^\star)}> 0; \\
(\overline{I_2})\qquad & 1+\frac{x^\star f_1'(x^\star)}{f_1(x^\star)}+\frac{x^\star f_2'(x^\star)}{f_2(x^\star)}\leq 0.
\end{align*}

\begin{remark}\label{rem.Hill.1}
Consider the feedback function
$$
f_1(u)=k_1\left(1-\eta\df{u^{\alpha}}{c^{\alpha}+u^{\alpha}}\right)
$$
with $k_1>0$, $\eta\in(0,1)$, $c>0$, $\alpha\geq 1$. A straightforward computation yields:
$$8f_1(u)+uf_1'(u)=k_1\frac{8(1-\eta)u^{2\alpha}+(16-8\eta-\alpha\eta)u^\alpha c^\alpha+8c^{2\alpha}}{(c^\alpha+u^\alpha)^2}.$$
It can be easily seen that if $\alpha\leq 8$, we have
$$8f_1(u)+uf_1'(u)\geq k_1\frac{8(1-\eta)u^{2\alpha}+16(1-\eta)u^\alpha c^\alpha+8c^{2\alpha}}{(c^\alpha+u^\alpha)^2}\geq 0,$$
and hence, the inequality $8f_1(u)+uf_1'(u)\geq 0$ is satisfied for any $u\geq 0$ (regardless of the choice of parameters $k_1$, $\eta$ or $c$).

Therefore,  if $\alpha\leq 8$, inequality $(I_1)$ holds as well. It is worth mentioning that according to \cite{Murray_2002}, $\alpha>7$ is considered unphysiological.
\end{remark}

\begin{remark}\label{rem.Hill.2}
Consider the particular case of feedback functions given by (\ref{func.hill}), with $k_1,k_2>0$, $\eta,\mu\in(0,1)$, $\alpha\geq 1$ and with the constant $c$ given by
$$c=x^\star.$$
This assumption comes from the fact that the constant $c$ is chosen to be equal to the mean value of free cortisol \cite{Vinther_2011}, which, in turn, is equal to the last component of the equilibrium point $E$. In this case, the term that appears in the left hand-side of inequalities $(I_2)$ and $(\overline{I_2})$ becomes
$$1+\frac{x^\star f_1'(x^\star)}{f_1(x^\star)}+\frac{x^\star f_2'(x^\star)}{f_2(x^\star)}=1-\frac{\alpha}{2}\left[\frac{\eta}{2-\eta}+\frac{\mu}{2-\mu}\right].$$
It is important to note that in this case, inequalities $(I_2)$ and $(\overline{I_2})$ only depend on the parameters $\alpha,\eta,\mu$, and do not depend on the choice of the parameters $k_1,k_2$. For instance, if $\eta=\mu=0.5$, inequality $(I_2)$ is equivalent to $\alpha<3$.
\end{remark}

\begin{theorem}[Local asymptotic stability]\label{thm.stab} $ $
\begin{enumerate}
\item In the non-delayed case, if inequality $(I_1)$ is satisfied, then the equilibrium point $E$ of system (\ref{sys.hpa.dd}) is locally asymptotically stable.
\item For any delay kernels $h_i(t)$, $i\in\{1,2,31,32\}$, if inequality $(I_2)$ holds, then the equilibrium point $E$ of system (\ref{sys.hpa.dd}) is locally asymptotically stable.
\end{enumerate}
\end{theorem}

\begin{proof} 1. In the non-delayed case, the characteristic equation (\ref{eq.char}) becomes:
\begin{equation}
\la^3+c_1\la^2+c_2\la+c_3=0,
\end{equation} where
\begin{align*}
c_1&=w_1+w_2+w_3>0, \\
c_2&=w_1w_2+w_2w_3+w_1w_3+a>0,\\
c_3&=w_1w_2w_3+aw_1+b>0.
\end{align*}
We can easily compute
$$
c_1c_2-c_3=a(w_2+w_3)+\gamma(w_1,w_2,w_3)+2w_1w_2w_3-b,
$$
where
$$
\gamma(w_1,w_2,w_3)=w_1^2w_2+w_1w_2^2+w_1^2w_3+w_1w_3^2+w_2w_3^2+w_2^2w_3.
$$
By the inequality of arithmetic and geometric means, we deduce
$$\gamma(w_1,w_2,w_3)\geq 6w_1w_2w_3,\quad \textrm{for any }w_1,w_2,w_3>0,$$
and hence, based on (\ref{eq.b}) and $(I_1)$, we obtain:
\begin{align*}
c_1c_2-c_3&\geq a(w_2+w_3)+8w_1w_2w_3-b=\\
&=a(w_2+w_3)+w_1w_2w_3\left(8+\df{x^\star f_1'(x^\star)}{f_1(x^\star)}\right)>0.
\end{align*}
By the Routh-Hurwitz stability test, the equilibrium point $E$ is asymptotically stable.

2. From (\ref{eq.a}) and (\ref{eq.b}) it can be easily seen that inequality $(I_2)$ is equivalent to
$$a+\df{b}{w_1}< w_2w_3.$$
The characteristic equation (\ref{eq.char}) can be written as $$\varphi(z)=\psi(z),$$
where the functions $\varphi$ and $\psi$ are given by
\begin{align*}
\varphi(z)&=-(z+w_1)(z+w_2)(z+w_3),\\
\psi(z)&=a(z+w_1)H_2(z)H_{32}(z)+bH_1(z)H_2(z)H_{31}(z).
\end{align*}
These functions are holomorphic in the right half-plane.

Let $z\in \mathbb{C}$ with $\Re(z)\geq 0$. For any $i\in\{1,2,31,32\}$, from (\ref{delay.kernel.properties}) we obtain:
\begin{align*}
|H_i(z)|=\left|\int_{0}^\infty e^{-z s}h_i(s)ds\right|\leq\int_0^\infty |e^{-z s}|h_i(s)ds=\int_0^\infty e^{-\Re(z) s}h_i(s)ds\leq \int_0^\infty h_i(s)ds=1,
\end{align*}
and hence, we have:
\begin{align*}
|\psi(z)|&\leq a|z+w_1||H_2(z)||H_{32}(z)|+b|H_1(z)||H_2(z)||H_{31}(z)|\leq a|z+w_1|+b\\
&=|z+w_1|\left(a+\frac{b}{|z+w_1|}\right)=|z+w_1|\left( a+\frac{b}{\sqrt{|z|^2+w_1^2+2\Re(z)w_1}}\right)\\
&\leq|z+w_1|\left( a+\frac{b}{w_1}\right)<|z+w_1| w_2w_3 \\
&\leq|z+w_1| \left[(|z|^2+w_2^2+2\Re(z)w_2)(|z|^2+w_3^2+2\Re(z)w_3)\right]^{1/2}\\
&=|z+w_1||z+w_2||z+w_3|=|\varphi(z)|.
\end{align*}
Hence, the inequality $|\psi(z)|<|\varphi(z)|$ holds for any $z\in \mathbb{C}$, $\Re(z)\geq 0$. Therefore, the characteristic equation $\varphi(z)=\psi(z)$ does not have any root in the right half-plane (or the imaginary axis). This means that all the roots of the characteristic equation (\ref{eq.char}) have strictly negative real part, and the equilibrium $E$ is asymptotically stable.
\end{proof}

\begin{corollary}
For any delay kernels $h_i(t)$, $i\in\{1,2,31,32\}$, if the equilibrium point $E$ of system (\ref{sys.hpa.dd}) is unstable, then inequality $(\overline{I_2})$ holds. In other words, inequality $(\overline{I_2})$ is a necessary condition for the occurrence of bifurcations in system (\ref{sys.hpa.dd}).
\end{corollary}

\begin{remark}\label{rem.Hill.3}
If $f_1$ is given by (\ref{func.hill}) with $\alpha\leq 8$, it follows from Remark \ref{rem.Hill.1} that inequality $(I_1)$ holds. Based on Theorem \ref{thm.stab}, the equilibrium point $E$ is asymptotically stable in the non-delayed case. This improves the sufficient condition $\alpha\leq 7.46$ presented in \cite{Vinther_2011}. As $\alpha>7$ is considered unphysiological \cite{Murray_2002}, the equilibrium point $E$ is locally asymptotically stable for all realistic values of the parameters if no time delays are considered.

Moreover, if $f_1,f_2$ are given by (\ref{func.hill}) (such as in Remark \ref{rem.Hill.2}), with $c=x^\star$, and the following inequality is satisfied
$$\frac{2}{\alpha}>\frac{\eta}{2-\eta}+\frac{\mu}{2-\mu},$$
the equilibrium point $E$ of system (\ref{sys.hpa.dd}) is asymptotically stable for any choice of the delay kernels $h_i$, $i\in\{1,2,31,32\}$ and of the parameters $k_1,k_2,k_3,w_1,w_2,w_3$.

In the special case $\alpha=1$ considered in \cite{Savic_2005,Savic_2006}, it is easy to see that the above inequality is fulfilled for any $\eta,\mu\in(0,1)$, implying that the equilibrium point $E$ is locally asymptotically stable for any choice of the delay kernels $h_i$, $i\in\{1,2,31,32\}$ and of the parameters $k_1,k_2,k_3,w_1,w_2,w_3$.
\end{remark}

\section{Bifurcation analysis of system (\ref{sys.hpa.dd})}

The bifurcation analysis presented in this section takes into consideration the average time delays of the delay kernels $h_i$, $i\in\{1,2,31,32\}$. A first observation is that the time required by CRH to travel from the hypothalamus to the pituitary through the hypophyseal portal blood vessels is extremely short \cite{Bairagi_2008} and therefore, in most numerical simulations the average time delay $\tau_1$ is considered close to $0$. Moreover, $\tau_{31}$ and $\tau_{32}$ are comparable, as they represent the average time delays due to the negative feedback effect of the adrenal glands on the hypothalamus and pituitary, respectively, which are closely situated.

For this reason, in this section, we will assume for simplicity that
$$
H_{32}(z)=H_1(z)H_{31}(z),
$$
and we denote
$$H(z)=H_2(z)H_{32}(z)=H_1(z)H_2(z)H_{31}(z).$$
In fact, $H(z)$ is the Laplace transform of the convolution of the delay kernels $h_2$ and $h_{32}$ defined as
$$h(t)=\int_0^t h_2(s)h_{32}(t-s)ds,$$
with the mean
\begin{equation}\label{eq.tau}
\tau=\int_0^{\infty}sh(s)ds=\tau_2+\tau_{32},
\end{equation}
where $\tau_2$ and $\tau_{32}$ represent the average delays of the kernels $h_2$ and $h_{32}$ respectively. This results from the fact that the probability density function of the sum of two independent random variables is the convolution of their separate probability density functions.

Therefore, the characteristic equation (\ref{eq.char}) becomes
\begin{equation}\label{eq.char.bif}
( z+w_1)( z+w_2)( z+w_3)+[a(z+w_1)+b]H(z)=0,
\end{equation}
or equivalently:
$$H(z)^{-1}=Q(z),$$
where
$$Q(z)=-\frac{aw_1+b+a z}{( z+w_1)( z+w_2)( z+w_3)}.$$
The function $Q(z)$ defined above will play an important role in the bifurcation analysis presented in this section. We summarize its most important properties in the following Lemma.

\begin{lemma}\label{lem.Q}
The function
$$\om\mapsto|Q(i\om)|=\sqrt{\frac{a^2\om^2+(aw_1+b)^2}{(\om^2+w_1^2)(\om^2+w_2^2)(\om^2+w_3^2)}}$$
is strictly decreasing on $[0,\infty)$ and the equation
$$|Q(i\om)|=1$$
has a unique positive real root $\om_0$ if and only if inequality $(\overline{I2})$ is satisfied.

Moreover, the following inequality holds:
$$\Im\left(\df{Q'(i\om)}{Q(i\om)}\right)>0\qquad\forall\,\om>0.$$
\end{lemma}
\begin{proof}
We have
$$
|Q(i\om)|^2=\frac{a^2}{(\om^2+w_2^2)(\om^2+w_3^2)}+\frac{2aw_1+b^2}{(\om^2+w_1^2)(\om^2+w_2^2)(\om^2+w_3^2)}
$$
and it is easy to see that $\om\mapsto|Q(i\om)|$ is strictly decreasing on $[0,\infty)$, approaching $0$ as $\om\rightarrow\infty$. Therefore, the equation $|Q(i\om)|=1$ has a unique solution if and only if $|Q(0)|>1$, or $w_1w_2w_3<aw_1+b$, which is equivalent to inequality $(\overline{I2})$.

Moreover, we have:
\begin{align*}
\frac{d}{d\om}|Q(i\om)|^2&=\frac{d}{d\om}\left[Q(i\om)\overline{Q(i\om)}\right]=2\Re\left[\overline{Q(i\om)}\frac{d}{d\om}Q(i\om)\right]\\
&=2\Re\left[i\overline{Q(i\om)}Q'(i\om)\right]=-2\Im\left[\overline{Q(i\om)}Q'(i\om)\right]\\
&=-2|Q(i\om)|^2\Im\left(\df{Q'(i\om)}{Q(i\om)}\right).
\end{align*}
As $\om\mapsto |Q(i\om)|^2$ is strictly decreasing on $(0,\infty)$, its derivative is strictly negative, and hence, $\Im\left(\df{Q'(i\om)}{Q(i\om)}\right)>0$, for any $\om>0$.
\end{proof}

\begin{remark} A simple computation shows that
$$\Re[Q(i\om)]=\frac{a\om^4+[b(w_1+w_2+w_3)+a(w_1^2-w_2w_3)]\om^2-w_1w_2w_3(aw_1+b)}{(\om^2+w_1^2)(\om^2+w_2^2)(\om^2+w_3^2)}.$$
This formula will be useful in the framework of the bifurcation results that follow.
\end{remark}

Due to the high complexity of the problem, we are unable to perform the bifurcation analysis for general kernels $h_i$, $i\in\{1,2,31,32\}$. Thus, we focus our attention on the following cases:
\begin{itemize}
\item[1.] all delay kernels are Dirac kernels;
\item[2.] all delay kernels are Gamma kernels;
\item[3.] some delay kernels are Dirac kernels while others are Gamma kernels.
\end{itemize}

\subsection{Dirac kernels}

Consider that all the delay kernels are Dirac kernels: $h_1(t)=\delta(t-\tau_1)$, $h_2(t)=\delta(t-\tau_2)$, $h_{31}(t)=\delta(t-\tau_{31})$, $h_{32}(t)=\delta(t-\tau_{32})$, where $\tau_1,\tau_2,\tau_{31},\tau_{32}\geq 0$ satisfy the property
\begin{equation}\label{cond.dirac}
\tau_2+\tau_{32}=\tau_1+\tau_2+\tau_{31}=\tau>0.
\end{equation}
In this case, the characteristic equation (\ref{eq.char.bif}) becomes:
\begin{equation}\label{eq.char.Dirac}
(z+w_1)( z+w_2)( z+w_3)+[a( z+w_1)+b]e^{-\tau z}=0,
\end{equation}
or equivalently:
$$e^{\tau z}=Q(z).$$
We choose $\tau$ as bifurcation parameter.

\begin{theorem}[Hopf bifurcations in the case of Dirac kernels]\label{thm.bif.dirac}
Assume that inequalities $(I_1)$ and $(\overline{I_2})$ are satisfied.
For any $p\in\mathbb{Z}^+$, consider
\begin{equation}\label{eq.tau.Dirac}
\tau_p=\frac{\arccos\left[\Re(Q(i\om_0))\right]+2p\pi}{\om_0},
\end{equation}
where $\om_0>0$ is given by Lemma \ref{lem.Q}. The equilibrium point $E$ is asymptotically stable if any only if $\tau\in[0,\tau_0)$. For any $p\in\mathbb{Z}^+$, at $\tau=\tau_p$, system (\ref{sys.hpa.dd}) undergoes a Hopf bifurcation at the equilibrium point $E$.
\end{theorem}

\begin{proof}
Equation (\ref{eq.char.Dirac}) has a pair of complex conjugated solutions $ z=\pm i\om$ on the imaginary axis ($\om>0$) if and only if
\begin{equation}\label{eq.1}
e^{i\om\tau}=Q(i\om).
\end{equation}
Taking the absolute value in (\ref{eq.1}) we obtain that $|Q(i\om)|=1$, and hence, we obtain that $\om=\om_0$, where $\om_0$ is the unique positive real solution given by Lemma \ref{lem.Q}.

From (\ref{eq.1}) we can also deduce
$$\cos(\tau\om_0)=\Re[Q(i\om_0)],$$
which leads us to the values $\tau_p$ given by (\ref{eq.tau.Dirac}).

From Theorem \ref{thm.stab}, we know that the equilibrium point $E$ is asymptotically stable if $\tau=0$ (since inequality $(I_1)$ holds). The number of the roots of the characteristic equation from the left half-plane can change only if a root (or pair of complex conjugated roots) crosses the imaginary axis, or more precisely, whenever $\tau=\tau_p$, $p\in\mathbb{Z}^+$ (in which case, $\pm i \omega_0$ are roots of the characteristic equation). Therefore, for any $\tau\in[0,\tau_0)$ the equilibrium point $E$ is asymptotically stable.

Let $ z_p(\tau)$ denote the root of the characteristic equation (\ref{eq.char.Dirac}) satisfying $ z_p(\tau_p)=i\om_0$. The function $z_p(\tau)$ satisfies $$e^{\tau  z_p(\tau)}=Q( z_p(\tau)).$$
Taking the derivative with respect to $\tau$, it follows that
$$(\tau z_p'(\tau)+ z_p(\tau))e^{\tau  z_p(\tau)}= z_p'(\tau)Q'( z_p(\tau)).$$
We obtain
$$ z_p'(\tau)=\frac{ z_p(\tau)}{Q'( z_p(\tau))e^{-\tau  z_p(\tau)}-\tau}=\frac{ z_p(\tau)}{\frac{Q'( z_p(\tau))}{Q( z_p(\tau))}-\tau},$$
and hence:
$$ z_p'(\tau_p)=\frac{i\om_0}{\frac{Q'(i\om_0)}{Q(i\om_0)}-\tau_p}.$$
Taking the real part and using Lemma \ref{lem.Q}, we obtain:
$$\left.\frac{d\Re( z_p)}{d\tau}\right|_{\tau=\tau_p}=\frac{\om_0\Im\left(\frac{Q'(i\om_0)}{Q(i\om_0)}\right)}{\left|\frac{Q'(i\om_0)}{Q(i\om_0)}-\tau_p\right|^2}>0.$$
This nondegeneracy condition for the Hopf bifurcation shows that the equilibrium point $E$ can only be asymptotically stable if and only if $\tau\in[0,\tau_0)$ and that for any $p\in\mathbb{Z}^+$, at $\tau=\tau_p$, system (\ref{sys.hpa.dd}) undergoes a Hopf bifurcation at the equilibrium point $E$.
\end{proof}

\subsection{Gamma kernels}

We now consider that the delay kernels are Gamma distribution kernels: $h_1(t)=\df{t^{n_1-1}e^{-t/\beta}}{\beta^{n_1}(n_1-1)!}$, $h_2(t)=\df{t^{n_2-1}e^{-t/\beta}}{\beta^{n_2}(n_2-1)!}$, $h_{31}(t)=\df{t^{n_{31}-1}e^{- t/\beta}}{\beta^{n_{31}} (n_{31}-1)!}$,  $h_{32}(t)=\df{ t^{n_{32}-1}e^{-t/\beta}}{\beta^{n_{32}}(n_{32}-1)!}$, where $\beta>0$ and $n_1,n_2,n_{31},n_{32}\in\mathbb{Z}^+\setminus\{0\}$ satisfy:
$$n_2+n_{32}=n_1+n_2+n_{31}=n\geq 2.$$
The characteristic equation (\ref{eq.char}) becomes:
\begin{equation}\label{eq.char.Gamma}
( z+w_1)( z+w_2)( z+w_3)+\frac{a( z+w_1)+b}{(\beta z+1)^n}=0,
\end{equation}
or equivalently
$$(\beta z+1)^n=Q(z).$$
Choosing $\beta$ as bifurcation parameter, we obtain the following result:

\begin{theorem}[Hopf bifurcations in the case of Gamma kernels]\label{thm.bif.gamma}
Assume that inequalities $(I_1)$ and $(\overline{I_2})$ are satisfied.
Let $\om_n$ denote the largest real root of the equation
\begin{equation}\label{eq.omega.Gamma}
T_n\left(\df{1}{|Q(i\om)|^{1/n}}\right)=\frac{\Re(Q(i\om))}{|Q(i\om)|}
\end{equation}
from the interval $(0,\om_0)$, where $T_n$ is the Chebyshev polynomial of the first kind of order $n$, and consider
\begin{equation}\label{eq.beta.Gamma}
\beta_n=\frac{1}{\om_n}\sqrt{|Q(i\om_n)|^{2/n}-1}.
\end{equation}
The equilibrium point $E$ is asymptotically stable if $\beta\in(0,\beta_n)$. At $\beta=\beta_n$, system (\ref{sys.hpa.dd}) undergoes a Hopf bifurcation at the equilibrium point $E$.
\end{theorem}

\begin{proof}
Equation (\ref{eq.char.Gamma}) has a pair of complex conjugated solutions $ z=\pm i\om$ on the imaginary axis ($\om>0$) if and only if
\begin{equation}\label{eq.2}
(i\beta\om+1)^n=Q(i\om).
\end{equation}
Applying the modulus to both sides of equation, we obtain:
$$(\beta^2\om^2+1)^n=|Q(i\om)|^2,$$
which means that $|Q(i\om)|>1$. Based on Lemma \ref{lem.Q}, we deduce that $\om\in(0,\om_0)$ and:
\begin{equation}\label{eq.3}
\beta=\frac{1}{\om}\sqrt{|Q(i\om)|^{2/n}-1}.
\end{equation}
Replacing (\ref{eq.3}) in (\ref{eq.2}) we can also deduce
$$\left(\frac{1}{|Q(i\om)|^{1/n}}+i\sqrt{1-\frac{1}{|Q(i\om)|^{2/n}}}\right)^n=\frac{Q(i\om)}{|Q(i\om)|}.$$
Denoting $\theta(\om)=\arccos\left(\df{1}{|Q(i\om)|^{1/n}}\right)\in\left(0,\df{\pi}{2}\right)$, the above equation becomes:
$$\cos(n\theta(\om))+i\sin(n\theta(\om))=\frac{Q(i\om)}{|Q(i\om)|}.$$
Taking the real part in the above equation, we obtain:
$$\cos(n\theta(\om))=\frac{\Re(Q(i\om))}{|Q(i\om)|},$$
or equivalently
$$T_n\left(\df{1}{|Q(i\om)|^{1/n}}\right)=\frac{\Re(Q(i\om))}{|Q(i\om)|},$$
where $T_n$ denotes the Chebyshev polynomial of the first kind of order $n$. Therefore, we have obtained equation (\ref{eq.omega.Gamma}).

As inequality $(I_1)$ holds, from the proof of the first part of Theorem \ref{thm.stab}, we know that if $\beta=0$, all the roots of the characteristic equation (\ref{eq.char.Gamma}) are in the left half-plane. The number of the roots of the characteristic equation from the left half-plane can change only if a root (or pair of complex conjugated roots) crosses the imaginary axis. From (\ref{eq.3}) and Lemma \ref{lem.Q} we can easily see that the critical values of $\beta$ decrease with respect to $\omega$, and hence, the smallest critical value of $\beta$ will be the one corresponding to the largest root $\om_n$ of equation (\ref{eq.omega.Gamma}) from the interval $(0,\om_0)$. Therefore, we obtain the smallest critical value $\beta_n$ of the bifurcation parameter, given by (\ref{eq.beta.Gamma}), and we deduce that for any $\beta\in(0,\beta_n)$ the equilibrium point $E$ is asymptotically stable.

Let $z(\beta)$ denote the root of the characteristic equation (\ref{eq.char.Gamma}) satisfying $ z(\beta_n)=i\om_n$. Based on the characteristic equation (\ref{eq.char.Gamma}), we obtain:
$$(\beta z(\beta)+1)^n=Q(z(\beta)).$$
Taking the derivative with respect to $\beta$, it follows that
$$n(\beta z(\beta)+1)^{n-1}(\beta z'(\beta)+ z(\beta))= z'(\beta)Q'( z(\beta)).$$
We obtain
$$ z'(\beta)=\frac{n z(\beta)}{(\beta z(\beta)+1)\frac{Q'( z(\beta))}{Q( z(\beta))}-n\beta},$$
and hence:
$$ z'(\beta_n)=\frac{i n\om_n}{(i\beta_n\om_n+1)\frac{Q'(i\om_n)}{Q(i\om_n)}-n\beta_n}.$$
Taking the real part, we obtain:
$$\left.\frac{d\Re( z)}{d\beta}\right|_{\beta=\beta_n}
=\frac{n\om_n\Im\left((i\beta_n\om_n+1)\frac{Q'(i\om_n)}{Q(i\om_n)}\right)}{\left|(i\beta_n\om_n+1)\frac{Q'(i\om_n)}{Q(i\om_n)}-n\beta_n\right|^2}=
\frac{n\om_n\left[\beta_n\om_n\Re\left(\frac{Q'(i\om_n)}{Q(i\om_n)}\right)+\Im\left(\frac{Q'(i\om_n)}{Q(i\om_n)}\right)\right]}{\left|(i\beta_n\om_n+1)\frac{Q'(i\om_n)}{Q(i\om_n)}-n\beta_n\right|^2}.$$
A laborious computation shows that the term $\beta_n\om_n\Re\left(\frac{Q'(i\om_n)}{Q(i\om_n)}\right)+\Im\left(\frac{Q'(i\om_n)}{Q(i\om_n)}\right)$ is positive, and hence, $\left.\df{d\Re( z)}{d\beta}\right|_{\beta=\beta_n}>0$, implying that the equilibrium point $E$ is asymptotically stable if  $\beta\in(0,\beta_n)$. System (\ref{sys.hpa.dd}) undergoes a Hopf bifurcation at the equilibrium point $E$ at $\beta=\beta_n$.
\end{proof}

\subsection{A mix of Dirac and Gamma kernels}

For simplicity, we will assume that the delay kernels $h_1,h_2,h_{31},h_{32}$ are either Dirac or Gamma kernels, such that
$$H_2( z)H_{32}( z)=H_1( z) H_2( z)H_{31}( z)=\frac{e^{-\tau z}}{(\beta z+1)^n}$$
where $\tau\geq 0$, $\beta>0$ and $n\in\mathbb{Z}^+\setminus\{0\}$.

The characteristic equation (\ref{eq.char}) becomes:
\begin{equation}\label{eq.char.mix}
( z+w_1)( z+w_2)( z+w_3)+[a( z+w_1)+b]\frac{e^{-\tau z}}{(\beta z+1)^n}=0,
\end{equation}
or equivalently
$$e^{\tau z}=\frac{Q(z)}{(\beta z+1)^n}.$$
Choosing $\tau$ as bifurcation parameter, we obtain the following result:

\begin{theorem}[Hopf bifurcations for a mix of delay kernels]\label{thm.bif.mix}
Assume that inequalities $(I_1)$ and $(\overline{I_2})$ are satisfied and that $\beta\in(0,\beta_n)$, where $\beta_n$ is given by (\ref{eq.beta.Gamma}).
Let $\tilde{\om}_{n}$ denote the unique real positive root of the equation
\begin{equation}\label{eq.omega.mix}
|Q(i\om)|^2=(\beta^2\om^2+1)^n
\end{equation}
and consider
\begin{equation}\label{eq.tau.mix}
\tilde{\tau}_{np}=\frac{1}{\tilde{\om}_{n}}\left[\arccos\left(\Re\left(\frac{Q(i\tilde{\om}_{n})}{(i\beta\tilde{\om}_{n}+1)^n}\right)\right)+2p\pi \right].
\end{equation}
The equilibrium point $E$ is asymptotically stable if and only if $\tau\in[0,\tilde{\tau}_{n0})$. At $\tau=\tilde{\tau}_{np}$, system (\ref{sys.hpa.dd}) undergoes a Hopf bifurcation at the equilibrium point $E$.
\end{theorem}

\begin{proof}
Equation (\ref{eq.char.mix}) has a pair of complex conjugated solutions $ z=\pm i\om$ on the imaginary axis ($\om>0$) if and only if
\begin{equation}\label{eq.4}
e^{i\tau\om }=\frac{Q(i\om)}{(i\beta \om+1)^n}.
\end{equation}
Applying the modulus to both sides of equation, we obtain:
$$|Q(i\om)|^2=(\beta^2\om^2+1)^n.$$
Based on Lemma \ref{lem.Q}, the left hand-side of this equation is strictly decreasing, while the right hand-side is strictly increasing on $(0,\infty)$, implying that there is a unique solution $\tilde{\om}_n\in(0,\om_0)$.

Taking the real parts of both sides of equation (\ref{eq.4}), we deduce the critical values of the bifurcation parameter $\tau$, given by (\ref{eq.tau.mix}).

For $\tau=0$, taking into account that $\beta\in(0,\beta_n)$, from Theorem \ref{thm.bif.gamma} we know that the equilibrium point $E$ is asymptotically stable. The number of the roots of the characteristic equation from the left half-plane can change only if a root (or pair of complex conjugated roots) crosses the imaginary axis, i.e. at the critical values $\tilde{\tau}_{np}$.

Let $z_{np}(\tau)$ denote the root of the characteristic equation (\ref{eq.char.mix}) satisfying $ z_{np}(\tilde{\tau}_{np})=i\tilde{\om}_n$. Based on the characteristic equation (\ref{eq.char.mix}), we obtain:
$$(\beta z_{np}(\tau)+1)^n e^{\tau z_{np}(\tau)}=Q(z_{np}(\tau)).$$
Taking the derivative with respect to $\tau$, it follows that
\begin{align*}
&n(\beta z_{np}(\tau)+1)^{n-1}\beta z_{np}'(\tau)e^{\tau z_{np}(\tau)}+(\beta z_{np}(\tau)+1)^n (\tau z_{np}'(\tau)+z_{np}(\tau))e^{\tau z_{np}(\tau)}\\
&= z_{np}'(\tau)Q'( z_{np}(\tau)).
\end{align*}
We obtain
$$ z_{np}(\tau)=\frac{z_{np}(\tau)}{\frac{Q'( z_{np}(\tau))}{Q( z_{np}(\tau))}-\tau-\frac{n\beta}{\beta z_{np}(\tau)+1}},$$
and hence:
$$ z_{np}'(\tilde{\tau}_{np})=\frac{i \tilde{\om}_n}{\frac{Q'( i \tilde{\om}_n)}{Q( i \tilde{\om}_n)}-\tau-\frac{n\beta}{i\beta \tilde{\om}_n+1}}.$$
Taking the real part, we obtain:
$$\left.\frac{d\Re( z_{np})}{d\tau}\right|_{\tau=\tilde{\tau}_{np}}
=\frac{\tilde{\om}_n\left[\Im\left(\frac{Q'( i \tilde{\om}_n)}{Q( i \tilde{\om}_n)}\right)+\frac{n\beta^2\tilde{\om}_n}{\beta^2\tilde{\om}_n^2+1}\right]}{\left|\frac{Q'( i \tilde{\om}_n)}{Q( i \tilde{\om}_n)}-\tau-\frac{n\beta}{i\beta \tilde{\om}_n+1}\right|^2}>0.$$
The positivity follows from Lemma \ref{lem.Q}. Thus, the equilibrium point $E$ can only be asymptotically stable if and only if $\tau\in[0,\tilde{\tau}_{n0})$ and for any $p\in\mathbb{Z}^+$, at $\tau=\tilde{\tau}_{np}$, system (\ref{sys.hpa.dd}) undergoes a Hopf bifurcation at the equilibrium point $E$.
\end{proof}

\section{A fractional-order model of the HPA axis}

In this section, the following fractional-order mathematical model of the HPA axis will be briefly discussed:
\begin{equation}\label{sys.hpa.frac}
\left\{\begin{array}{l}
^cD^{q} x_1(t)=f_1(x_3(t-\tau_{31}))-w_1x_1(t)\\
^cD^{q} x_2(t)=f_2(x_3(t-\tau_{32}))x_1(t-\tau_1)-w_2x_2(t)\\
^cD^{q} x_3(t)=k_3x_2(t-\tau_2)-w_3x_3(t)
\end{array}\right.
\end{equation}
where $q\in(0,1)$ and $\tau_1,\tau_2,\tau_{31},\tau_{32}\geq 0$ represent discrete time delays. The classical integer order derivative is generalized by the fractional-order Caputo derivative, defined below \cite{Kilbas,Lak,Podlubny}.

\begin{definition}
For a continuous function $x(t)$, with $x'\in L^1_{loc}(\mathbb{R}^+)$, the Caputo fractional-order derivative of order $q\in(0,1)$ of $f$ is defined by
$$
^cD^q x(t)=\frac{1}{\Gamma(1-q)}\int_0^t(t-s)^{-q}x'(s)ds.
$$
\end{definition}

It is obvious that the fractional-order system (\ref{sys.hpa.frac}) and the integer-order system (\ref{sys.hpa.dd})  have the same equilibrium state $E$. In order to study the stability of the equilibrium state $E$ in the framework of system (\ref{sys.hpa.frac}) without delays, we rely on the linearization theorem recently proved in \cite{Li_Ma_2013}. This linearization theorem is an analogue of the classical Hartman theorem for nonlinear integer-order dynamical systems. Moreover, for the corresponding linearized fractional-order system, the following stability result holds \cite{Matignon}:

\begin{theorem}\label{thm.matignon}
The linear fractional-order autonomous system
$$^cD^q \bold{x}=A \bold{x}\qquad\textrm{where}~~ A\in\mathbb{R}^{n\times n}$$
where $q\in(0,1)$ is asymptotically stable if and only if
\begin{equation}
\label{eq.lambda.spec}
|\arg(\lambda)|>\frac{q\pi}{2}\qquad\forall \lambda\in\sigma(A),
\end{equation}
where $\sigma(A)$ denotes the spectrum of the matrix $A$ (i.e. the set of all eigenvalues).
\end{theorem}

The above theorem shows that in the case of linear fractional-order systems, the necessary and sufficient conditions for the asymptotic stability of the equilibrium state are weaker than the corresponding conditions from the classical integer-order case. Therefore, taking into account Theorem \ref{thm.stab}, we can easily obtain the following result:

\begin{proposition}
In the non-delayed case ($\tau_1=\tau_2=\tau_{31}=\tau_{32}=0$), if inequality $(I_1)$ is satisfied, the equilibrium state $E$ of the fractional order system (\ref{sys.hpa.frac}) is locally asymptotically stable.
\end{proposition}

At this point, the general stability and bifurcation theory for nonlinear fractional-order systems with discrete delays is still an open problem, and an active area of research. Because of the lack of theoretical tools, we have to rely on numerical simulations to exemplify the existence of oscillatory solutions of system (\ref{sys.hpa.frac}), which will be presented in the next section.

\section{Numerical results and discussions}

\subsection{Parameter values}

For the numerical simulations, the values of the elimination constants $w_i$, $i\in\{1,2,3\}$ are computed according to the formula $w_i=\df{\ln(2)}{T_i}$, where $T_i$ represent the plasma half-life of hormones, and are given by: $T_1\approx 4$ min, $T_2\approx 19.9$ min, $T_3\approx 76.4$ min \cite{Carroll_2007,Vinther_2011}.

The equilibrium point $E$ of the system should be at the mean values of the hormones: $\bar{x}_1=7.659$ pg/ml (24-h mean value of CRH), $\bar{x}_2=21$ pg/ml (24-h mean value of ACTH) and $\bar{x}_3=3.055$ ng/ml (24-h mean value of free CORT) \cite{Carroll_2007}. Based on equation (\ref{eq.state}), this leads us to the following relationship:
$$\left(\df{f_1(x^\star)}{w_1}, \df{w_3x^\star}{k_3},x^\star\right)=(\bar{x}_1,\bar{x}_2,\bar{x}_3)$$
and hence:
\begin{align*}
x^\star&=\bar{x}_3=3.055\textrm{ ng/ml};\\
k_3&=w_3\frac{x^\star}{\bar{x}_2}= \frac{\ln(2)}{76.4\textrm{ min}}\cdot\frac{3.055\textrm{ ng/ml}}{21 \textrm{pg/ml}}=1.31985\textrm{ min}^{-1};\\
f_1(x^\star)&=w_1 \bar{x}_1= \frac{\ln(2)\cdot 7.659 \textrm{ pg/ml}}{4\textrm{ min}}=1.3272 \frac{\textrm{pg}}{\textrm{ml}\cdot\textrm{min}}.
\end{align*}
Moreover, due to the fact that $x^\star$ is the fixed point of the function $\df{k_3}{w_1w_2w_3}f_1(x)f_2(x)$ (see (\ref{sys.eq.points})), it follows that
$$f_2(x^\star)=\frac{w_1w_2w_3}{k_3}\cdot\frac{x^\star}{f_1(x^\star)}=0.0955 \textrm{ min}^{-1}.$$

For the numerical simulations, the feedback functions $f_1$ and $f_2$ are considered as in equation (\ref{func.hill}). For fixed values of the parameters $c$ (given in ng/ml, as $\bar{x}_3$) $\alpha\in[1,7]$, $\eta,\mu\in(0,1]$ (dimensionless), the parameters $k_1$ and $k_2$ are uniquely determined, based on the numerical values of $x^\star$, $f_1(x^\star)$ and $f_2(x^\star)$ determined above. Hence:
\begin{align*}
k_1&=\frac{1.3272}{1-\eta\frac{(3.055)^\alpha}{c^\alpha+(3.055)^\alpha}}\, \frac{\textrm{pg}}{\textrm{ml}\cdot\textrm{min}},\\
k_2&=\frac{0.0955}{1-\mu\frac{(3.055)^\alpha}{c^\alpha+(3.055)^\alpha}}\, \textrm{ min}^{-1}.
\end{align*}
In the following, we will assume for simplicity that $\eta=\mu$.


As for the average time delays, we first observe that the time required by CRH to travel from the hypothalamus to the pituitary through the hypophyseal portal blood vessels is extremely short \cite{Bairagi_2008} and therefore, for simplicity, we will assume a mean time delay $\tau_1=0$.

The human inhibitory time course concerning the negative feedback of cortisol on the production of ACTH shows great variability and has been described in the past as anything between 15 and 60 min \cite{Boscaro_1998,Posener_1997}. However, more recently, it has been shown that humans show fast HPA negative feedback \cite{Russell_2010}, suggesting that both GR (glucocorticoid receptors) and MR (mineralocorticoid receptors) are involved in this mechanism, with GR effecting a rapid nongenomic feedback at the level of the anterior pituitary and MR sensing higher glucocorticoid levels while levels are still rising \cite{Karst_2005}. Hence, we consider a mean delay $\tau_{32}\in(0,60]$.

In \cite{Hermus_1984}, a 30-min delay has been reported in the positive-feedforward effect of ACTH on plasma cortisol level, leading to the assumption that $\tau_2\in(0,30]$.

\subsection{Dirac kernels}

Based on the previous observations and (\ref{cond.dirac}), we choose the following discrete time delays:
\begin{itemize}
\item average time delay accounting for the positive feedback of the hypothalamus on the pituitary: $\tau_1=0$;
\item average time delay due to the positive feedback of the pituitary on the adrenal glands: $\tau_2\leq 30$ (min);
\item average time delay due to the negative feedback effect of the adrenal glands on the hypothalamus and pituitary, respectively: $\tau_{31}=\tau_{32}\leq 60$ (min).
\end{itemize}
 Our aim is to observe periodic solutions for sufficiently large values of the bifurcation parameter $\tau=\tau_2+\tau_{32}\leq 90$ (min), depending on the choice of the parameters $\alpha$, $\eta=\mu$ and $c$. Inequality $(\overline{I_2})$ has to be fulfilled, because it is a necessary condition for the occurrence of bifurcations. Based on Theorem \ref{thm.bif.dirac} and eq. (\ref{eq.tau.Dirac}), we can numerically determine the critical value $\tau_0$ of the bifurcation parameter corresponding to the occurrence of a Hopf bifurcation and we are looking for critical values which are smaller that $90$ (min). In Fig. \ref{fig:1}, for different values of $\alpha\in[2,7]$, we have represented the regions in the parameter plane $(c,\eta=\mu)$ for which inequality $(\overline{I_2})$ is satisfied, and the computed critical value $\tau_0$ is within $90$ min, $60$ min, $30$ min and $15$ min respectively.

 We observe that for $\alpha\geq 5$, for suitable choices of the parameters $c$ and $\eta=\mu$, the critical value $\tau_0$ is smaller than $15$ minutes. This means that periodic solutions can be obtained in system (\ref{sys.hpa.dd}) with discrete delays satisfying $\tau_0<\tau_2+\tau_{31}=\tau_2+\tau_{32}<15$ (min), which is in accordance with the fast feedback observed in humans \cite{Russell_2010}. For example, when $\alpha=6$, $\mu=\eta=1$ and $c=2$ (ng/ml), we compute $k_1=18.18$ (pg/(ml$\cdot$min)), $k_2=1.3$ (min$^{-1}$) and the critical value $\tau_0=11.4732$ (min). Therefore, periodic solutions can be observed for $\tau_2+\tau_{31}=\tau_2+\tau_{32}>11.4372$ (min) (see Fig. \ref{fig:dirac2}).

 However, if $\alpha<5$, we note that for any combination of parameters $c$ and $\eta=\mu$ we obtain $\tau_0>15$ (min). For example, when
$\alpha=3$, $\mu=\eta=0.95$ and $c=2$ (ng/ml), we compute $k_1=5.14$ (pg/(ml$\cdot$min)), $k_2=0.36$ (min$^{-1}$) and the critical value $\tau_0=46.5028$ (min). Hence, periodic solutions can only be observed for larger discrete delays (see Fig. \ref{fig:dirac1}).

Numerical simulations show that when the bifurcation parameter $\tau$ passes through the critical value $\tau_0$, oscillations appear due to Hopf bifurcation phenomena. We notice that oscillations corresponding to the case of smaller critical values (fast feedback) have higher amplitudes and higher frequency (over a 24 hour range) than those corresponding to the case of larger critical values (slow feedback).

In \cite{Vinther_2011}, where the minimal model of the HPA axis has been considered with discrete time delays, it has been reported that individual time delays need to exceed 19 min in order to observe oscillating solutions. Our numerical simulations show that, for a suitable choice of parameters, it is possible to obtain oscillations for time delays much smaller than 19 minutes, corresponding to the case of fast feedback noticed in \cite{Russell_2010}.

\begin{figure}[htbp]
\centering
\begin{tabular}{ccc}
\includegraphics[width=0.47\textwidth]{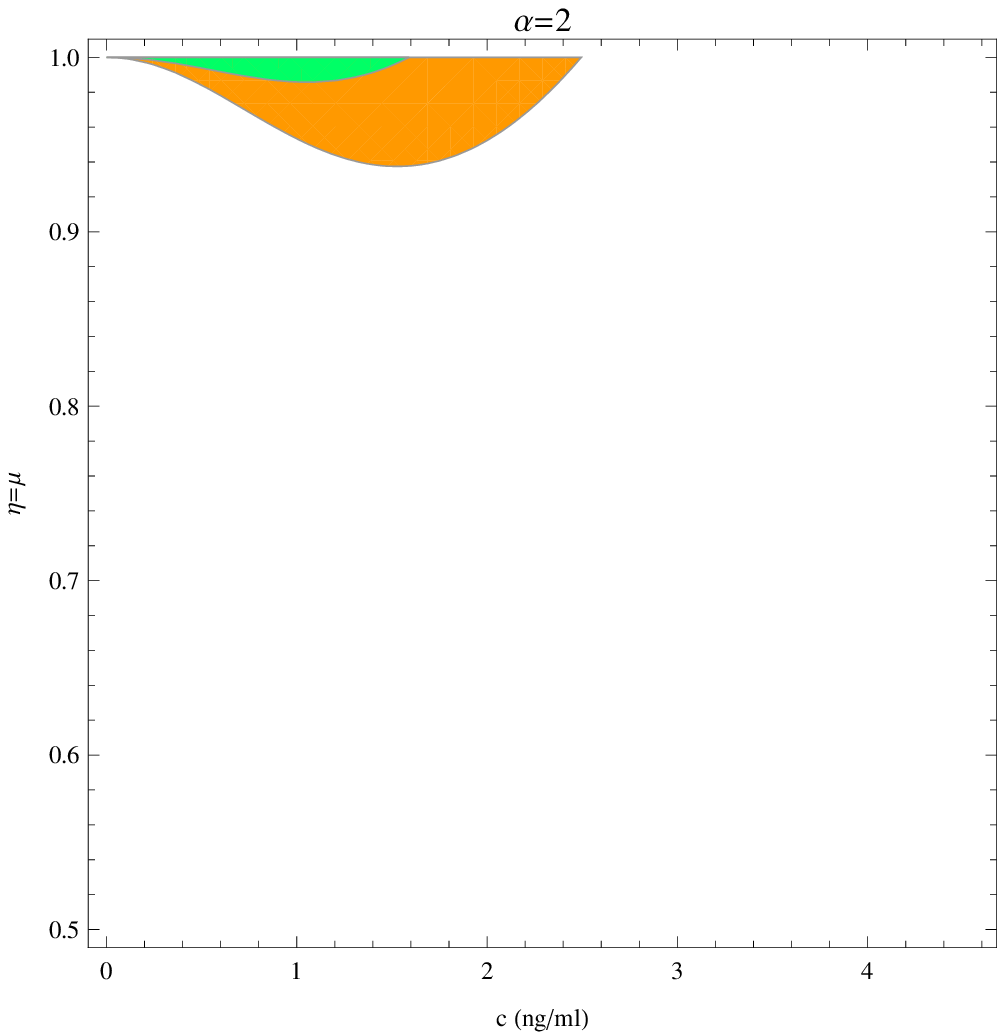} &
\includegraphics[width=0.47\textwidth]{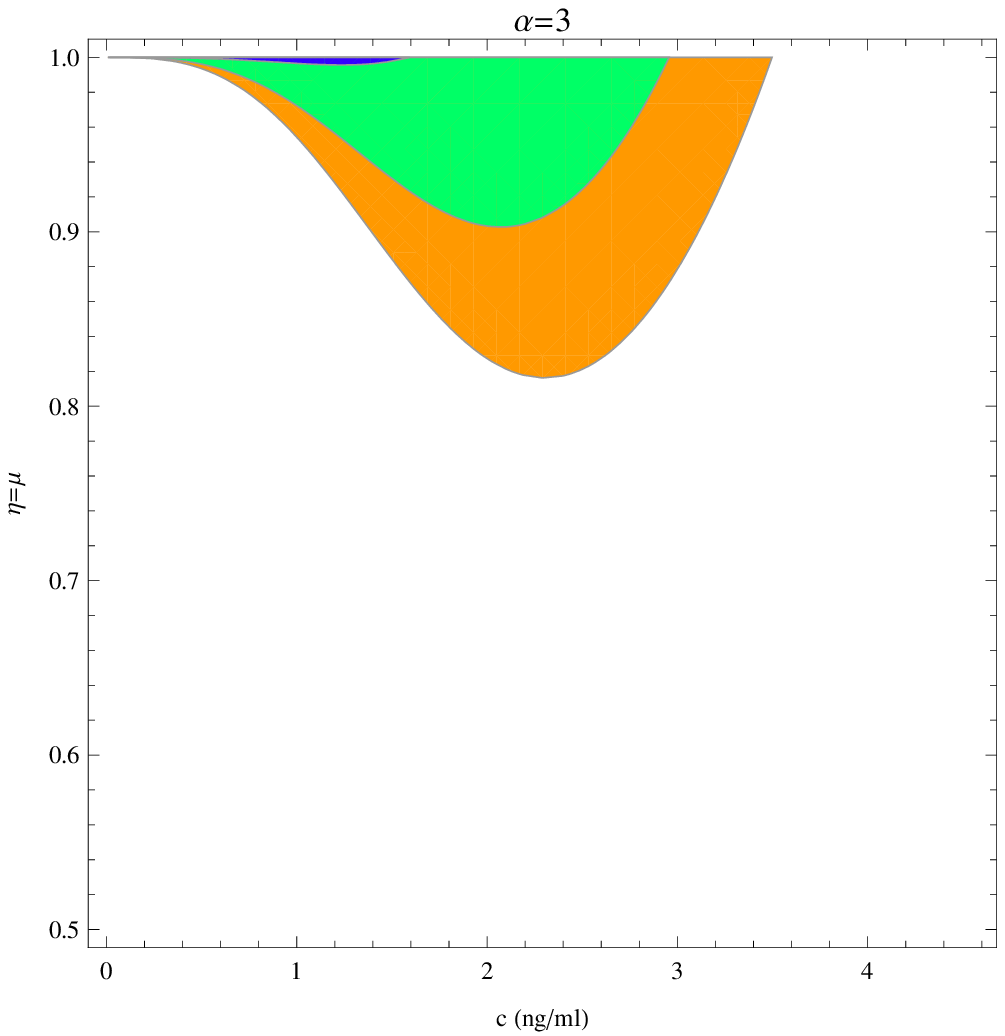} \\
\includegraphics[width=0.47\textwidth]{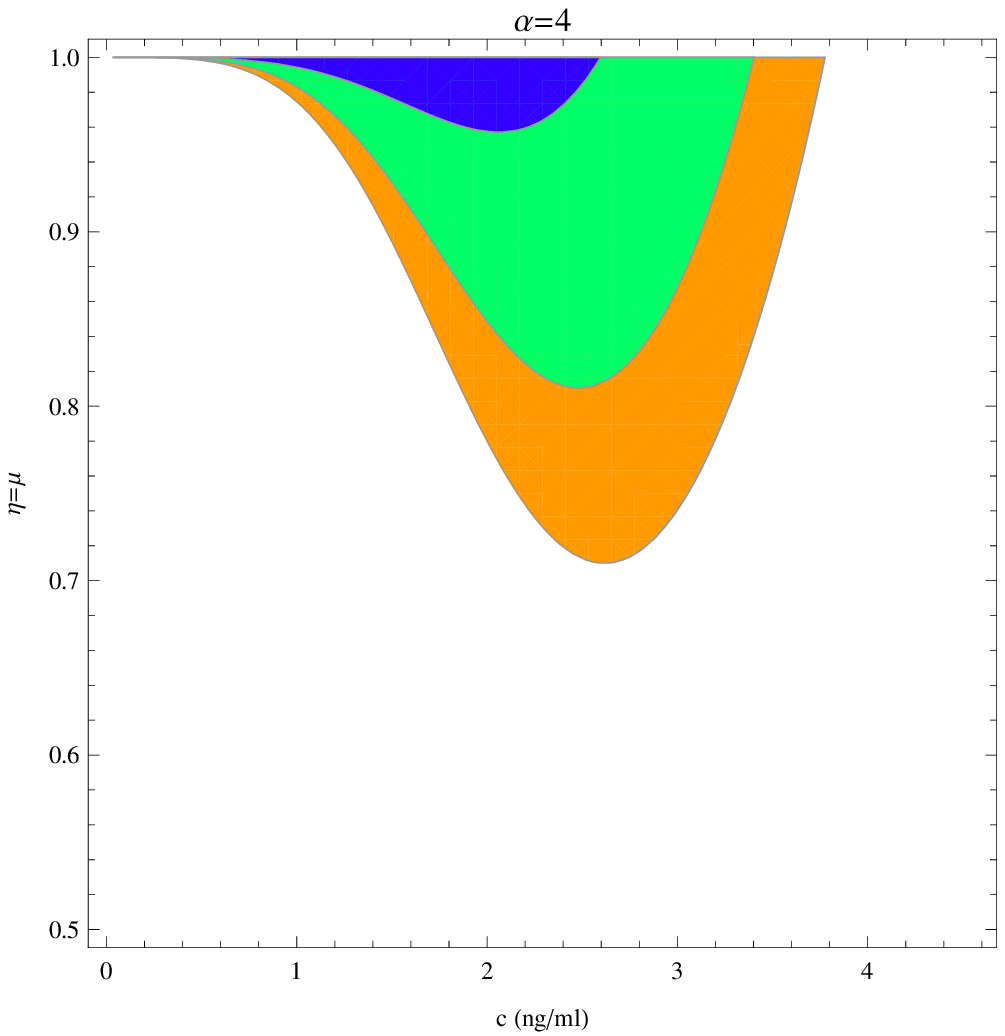} &
\includegraphics[width=0.47\textwidth]{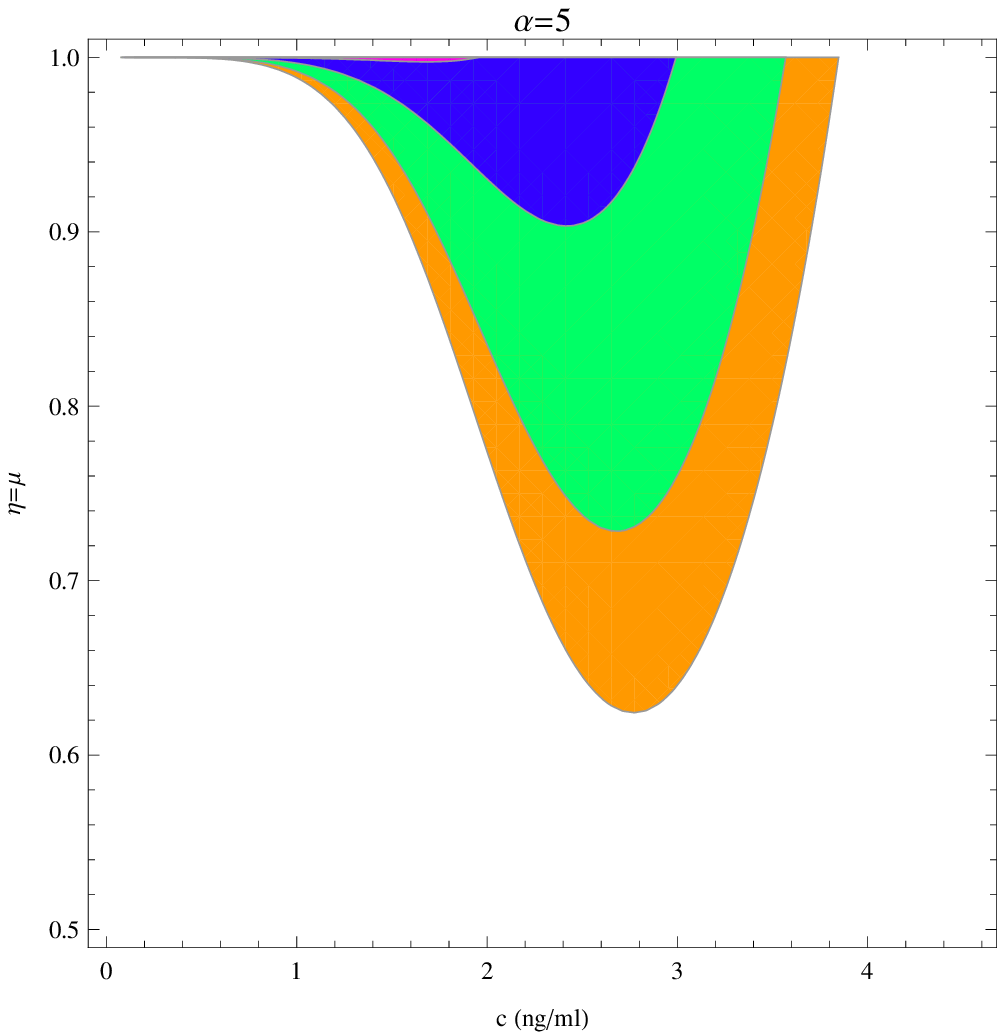} \\
\includegraphics[width=0.47\textwidth]{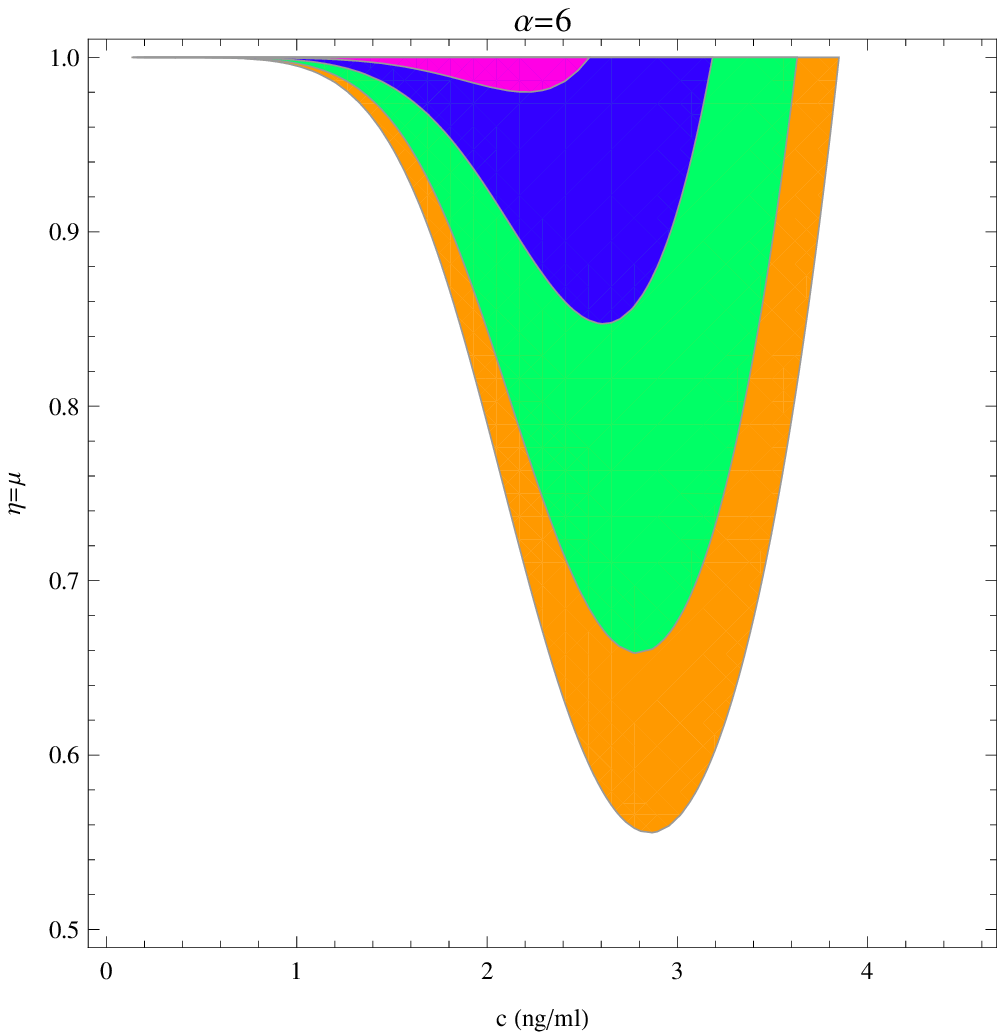} &
\includegraphics[width=0.47\textwidth]{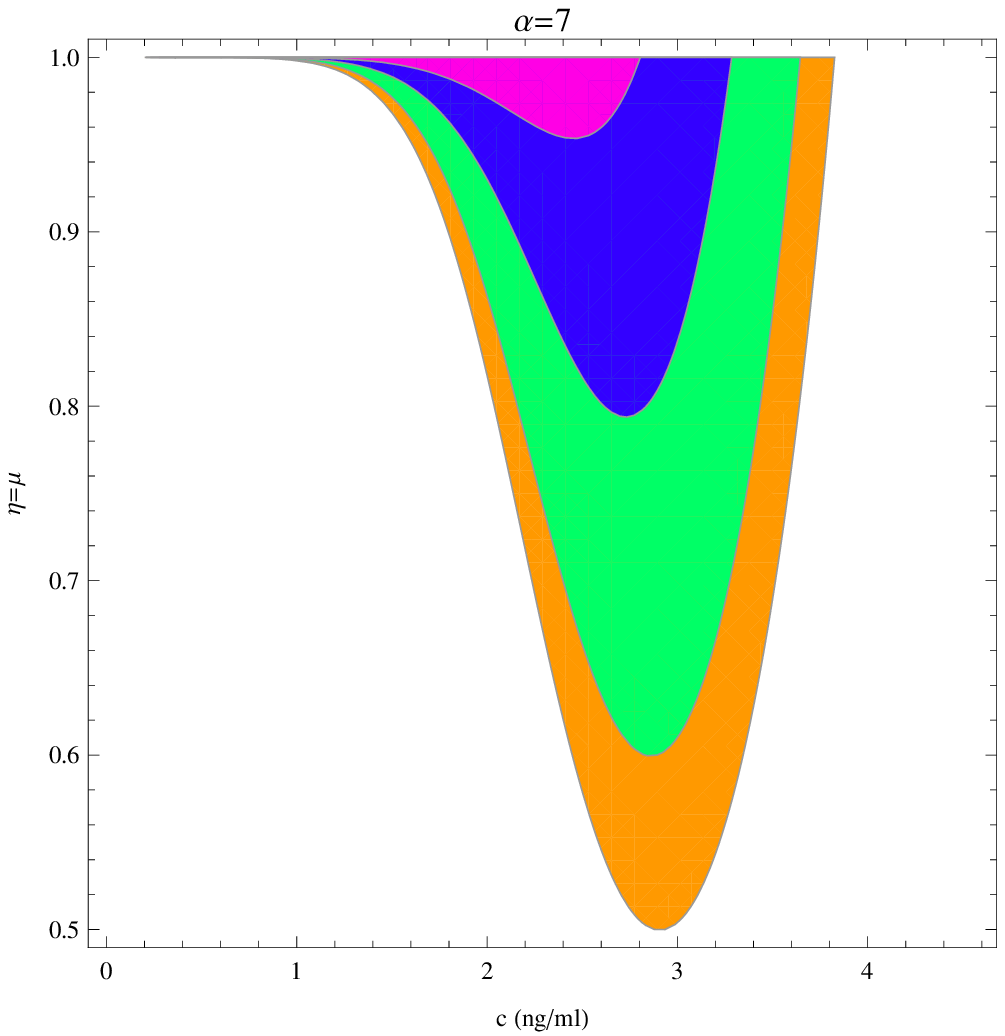}
\end{tabular}
\caption{Regions in the parameter plane $(c,\eta=\mu)$ resulting in critical values $\tau_0\in(60,90]$ (orange), $\tau_0\in(30, 60]$ (green), $\tau_0\in(15,30]$ (blue) and $\tau_0\in(0,15]$ (pink) respectively, while inequality $(\overline{I_2})$ is fulfilled, for different values of $\alpha$. (case of Dirac kernels)}
\label{fig:1}
\end{figure}

\begin{figure}
\centering
\begin{tabular}{ccc}
\includegraphics[width=0.47\textwidth]{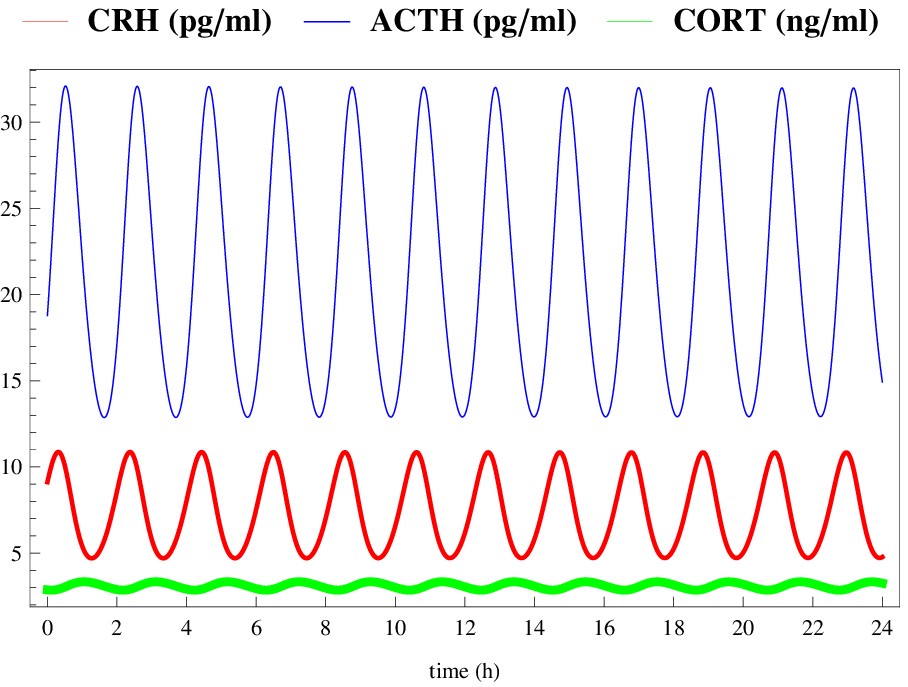} &
\includegraphics[width=0.47\textwidth]{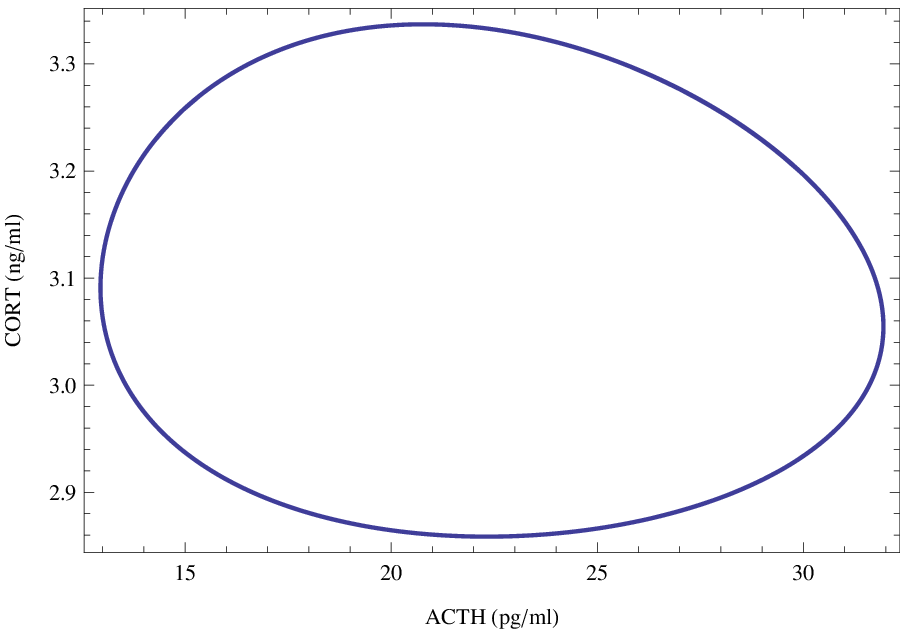} \\
\end{tabular}
\caption{Stable periodic orbit of system (\ref{sys.hpa.dd}) with Dirac kernels (discrete delays $\tau_1=0$, $\tau_2=5$ (min), $\tau_{31}=\tau_{32}=7$ (min)) due to  the Hopf bifurcation taking place when the bifurcation parameter $\tau$ exceeds the critical value $\tau_0=11.4732$ (min), in the case of feedback functions $f_1(u)=18.18\left(1-\frac{u^{6}}{2000^{6}+u^{6}}\right)$, $f_2(u)=1.3\left(1-\frac{u^{6}}{2000^{6}+u^{6}}\right)$.}
\label{fig:dirac2}
\end{figure}

\begin{figure}[htbp]
\centering
\begin{tabular}{ccc}
\includegraphics[width=0.47\textwidth]{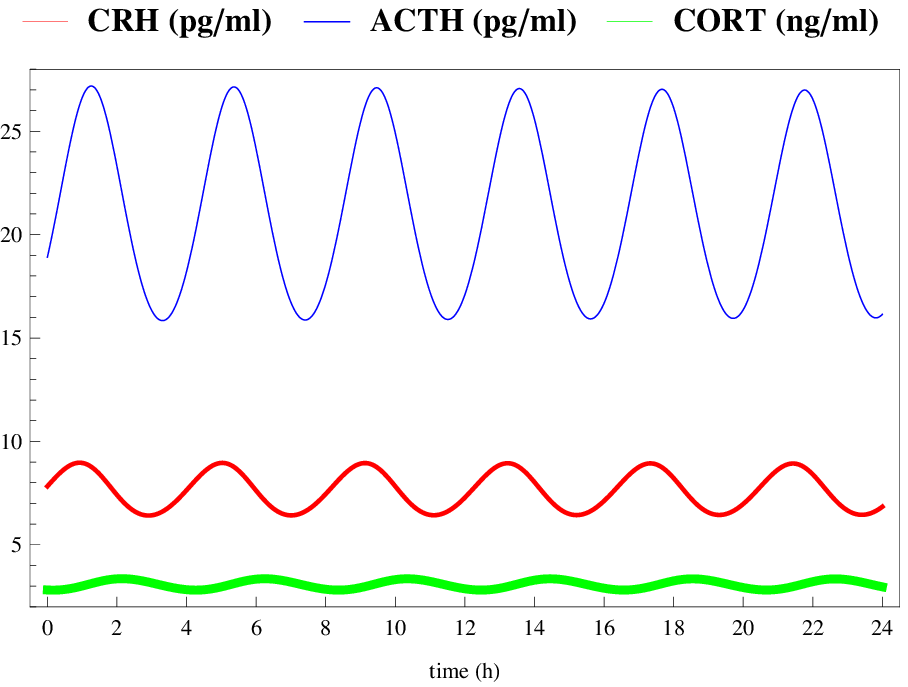} &
\includegraphics[width=0.47\textwidth]{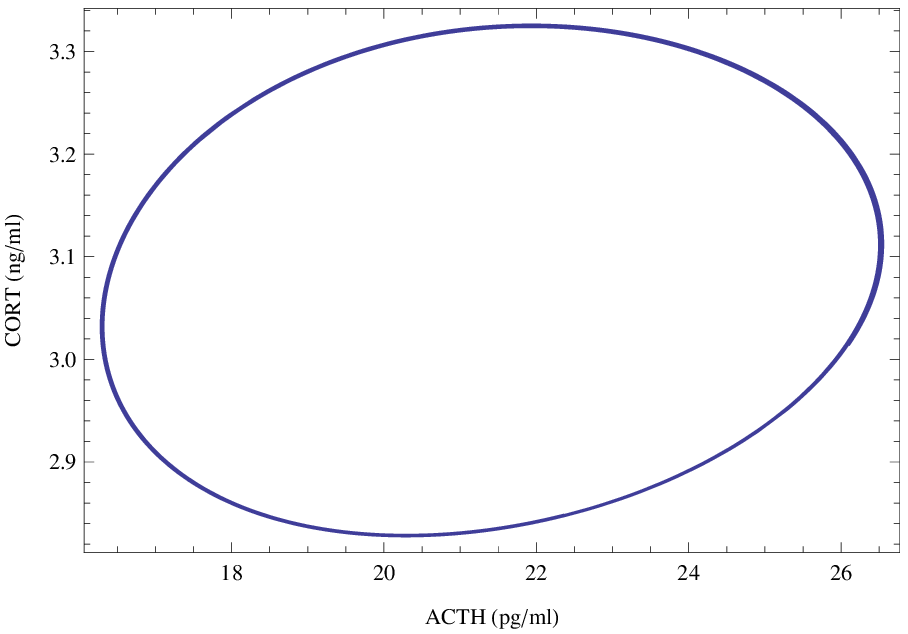} \\
\end{tabular}
\caption{Stable periodic orbit of system (\ref{sys.hpa.dd}) with Dirac kernels (discrete delays $\tau_1=0$, $\tau_2=27$ (min), $\tau_{31}=\tau_{32}=40$ (min)) due to  the Hopf bifurcation taking place when the bifurcation parameter $\tau$ exceeds the critical value $\tau_0=46.5028$ (min) in the case of feedback functions $f_1(u)=5.14\left(1-\frac{u^{3}}{2000^{3}+u^{3}}\right)$, $f_2(u)=0.36\left(1-\frac{u^{3}}{2000^{3}+u^{3}}\right)$.}
\label{fig:dirac1}
\end{figure}


\subsection{Gamma kernels}

For numerical simulations in the case of gamma kernels, assuming that there is no time delay accounting for the positive feedback of the hypothalamus on the pituitary, we consider $h_1(t)=\delta(t)$. Moreover, the other kernels are chosen to be strong gamma kernels (i.e. $n_2=n_{31}=n_{32}=2$):
$$h_2(t)=h_{31}(t)=h_{32}(t)=\df{t e^{-t/\beta}}{\beta^{2}}.$$
Strong kernels represent the case when the maximum influence on the growth rates of CRH, ACTH and CORT concentrations at any time $t$ is due to hormone concentrations at the previous time $t-\tau$, where $\tau=2\beta$ is the average time-delay. On the other hand, weak kernels (with  $n_2=n_{31}=n_{32}=1$) would indicate that the maximum weighted response of the growth rates is affected by the current hormone concentration level, while past concentrations have exponentially decreasing influence, which is less plausible than the case of strong kernels, from biological point of view.

In this case, $n=n_2+n_{32}=n_1+n_2+n_{31}=4$, and hence, the total average time-delay of the system is $\tau=4\beta$ (see eq. (\ref{eq.tau})).

Based on Theorem \ref{thm.bif.gamma} and eq. (\ref{eq.beta.Gamma}), we can numerically determine the critical value $\beta_4$  of the bifurcation parameter corresponding to the occurrence of a Hopf bifurcation. Therefore, we  can compute the critical value of the total average time-delay $\tau_g=4\beta_4$. Our aim is to find critical values satisfying $\tau_g\leq 90$ (min). In Fig. \ref{fig:gamma}, for different values of $\alpha\in[2,7]$, we have represented the regions in the parameter plane $(c,\eta=\mu)$ for which inequality $(\overline{I_2})$ is satisfied, and the computed critical value $\tau_g$ is within $90$ min, $60$ min, $30$ min and $15$ min respectively.

We observe that for $\alpha\geq 6$, for suitable choices of the parameters $c$ and $\eta=\mu$, the critical value $\tau_g$ is smaller than $15$ minutes. This means that periodic solutions can be obtained in system (\ref{sys.hpa.dd}) with gamma delay kernels satisfying $\tau=4\beta<15$ (min), which is in accordance with the fast feedback observed in humans \cite{Russell_2010}. When $\alpha=6$, $\mu=\eta=1$ and $c=2$ (ng/ml), providing $k_1=18.18$ (pg/(ml$\cdot$min)) and $k_2=1.3$ (min$^{-1}$), we compute the critical value of the bifurcation parameter $\beta_4=3.084$ (min) and the critical value of the total average time-delay $\tau_g=12.336$ (min). In Fig. \ref{fig:gamma2}, periodic solutions can be observed for $\beta=3.5$ (min).

However, if $\alpha<6$, we note that for any combination of parameters $c$ and $\eta=\mu$ we obtain $\tau_g>15$ (min). When
$\alpha=3$, $\mu=\eta=0.95$ and $c=2$ (ng/ml), providing $k_1=5.14$ (pg/(ml$\cdot$min)) and $k_2=0.36$ (min$^{-1}$), we determine the critical value $\beta_4=16.9753$ (min) and hence, $\tau_g=67.9$ (min). Periodic solutions can only be seen for $\beta>16.9753$ (min) (see Fig. \ref{fig:gamma1}).

Similarly as in the case of Dirac kernels, we notice that oscillations corresponding to smaller critical values (fast feedback) have higher amplitudes and higher frequency (over a 24 hour range) than those corresponding to larger critical values (slow feedback).

\begin{figure}[htbp]
\centering
\begin{tabular}{ccc}
\includegraphics[width=0.47\textwidth]{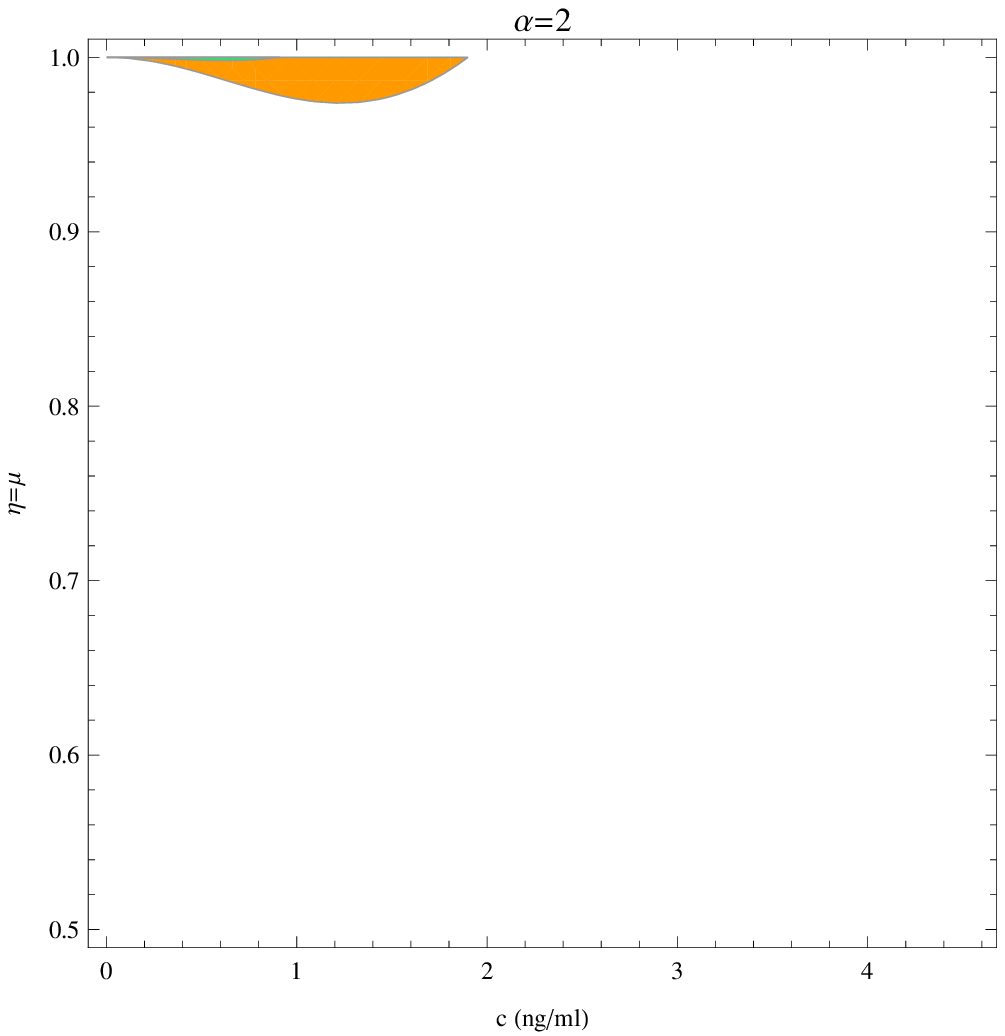} &
\includegraphics[width=0.47\textwidth]{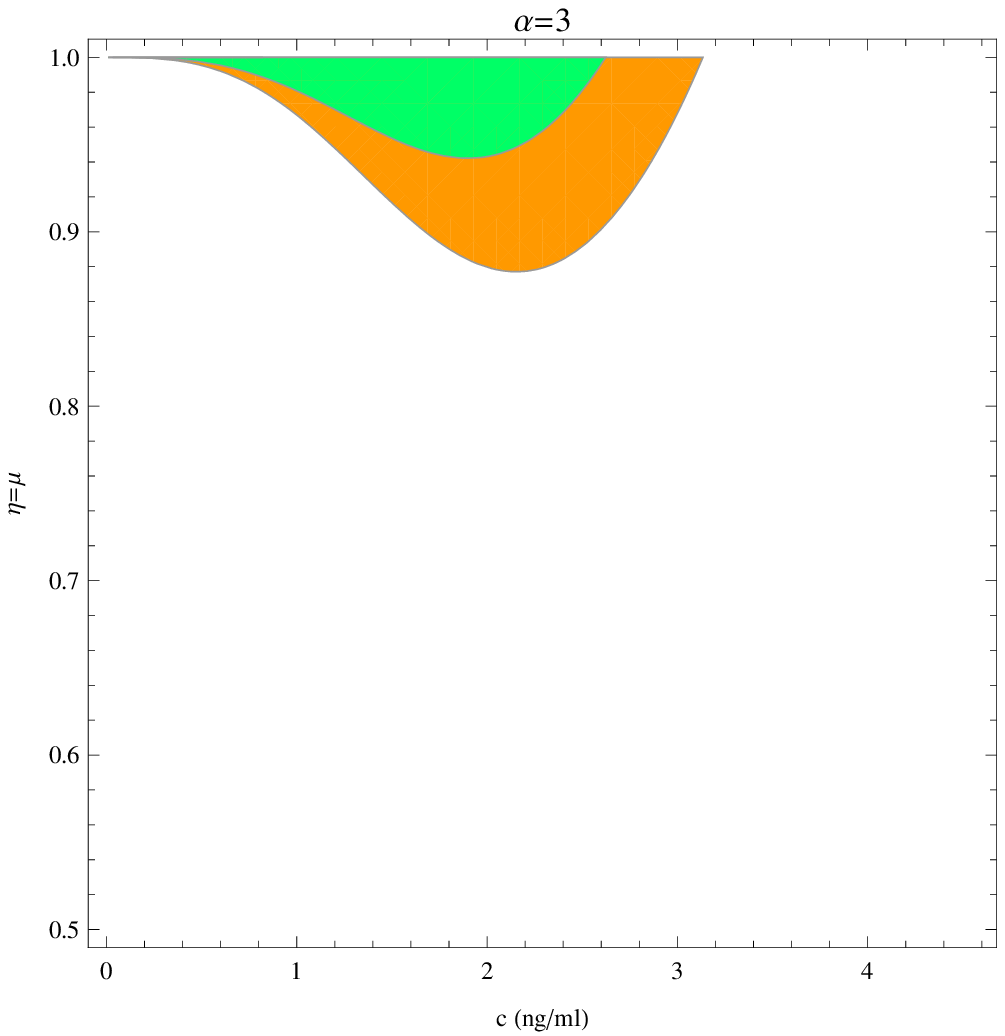} \\
\includegraphics[width=0.47\textwidth]{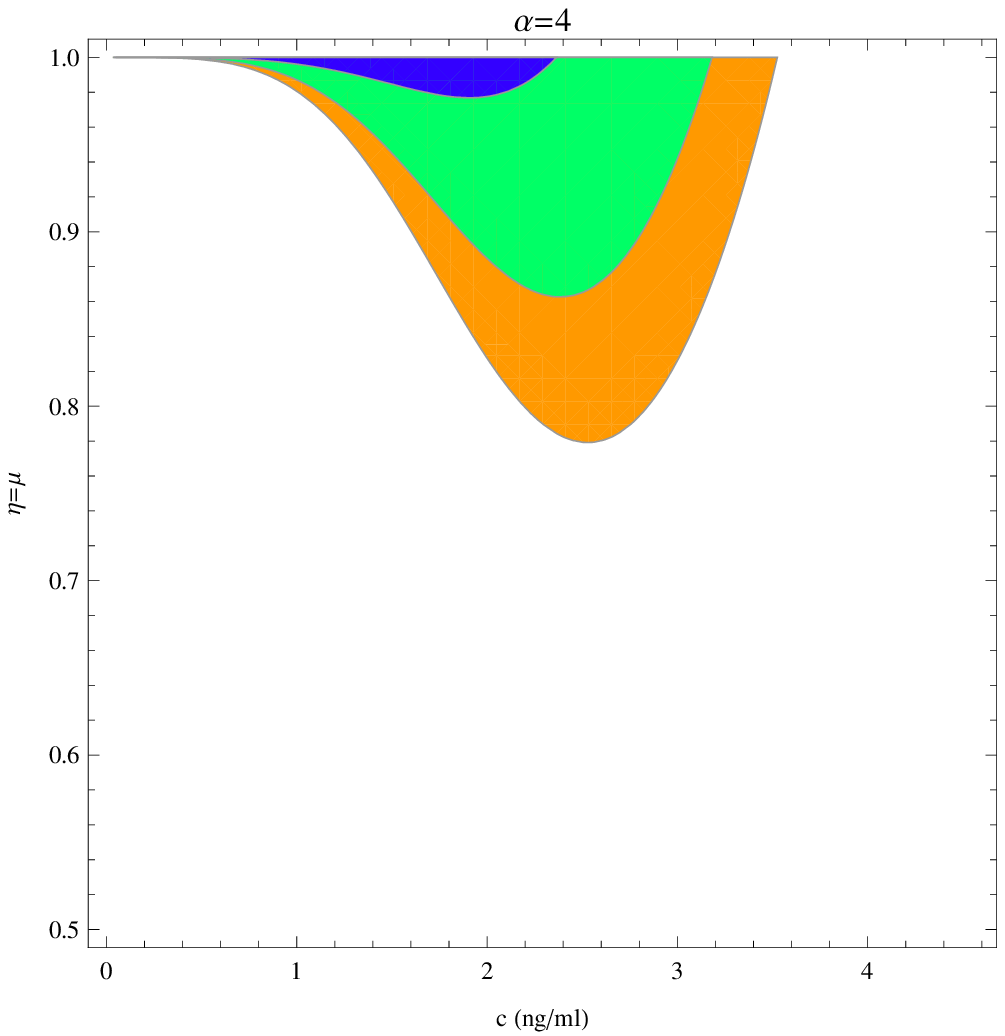} &
\includegraphics[width=0.47\textwidth]{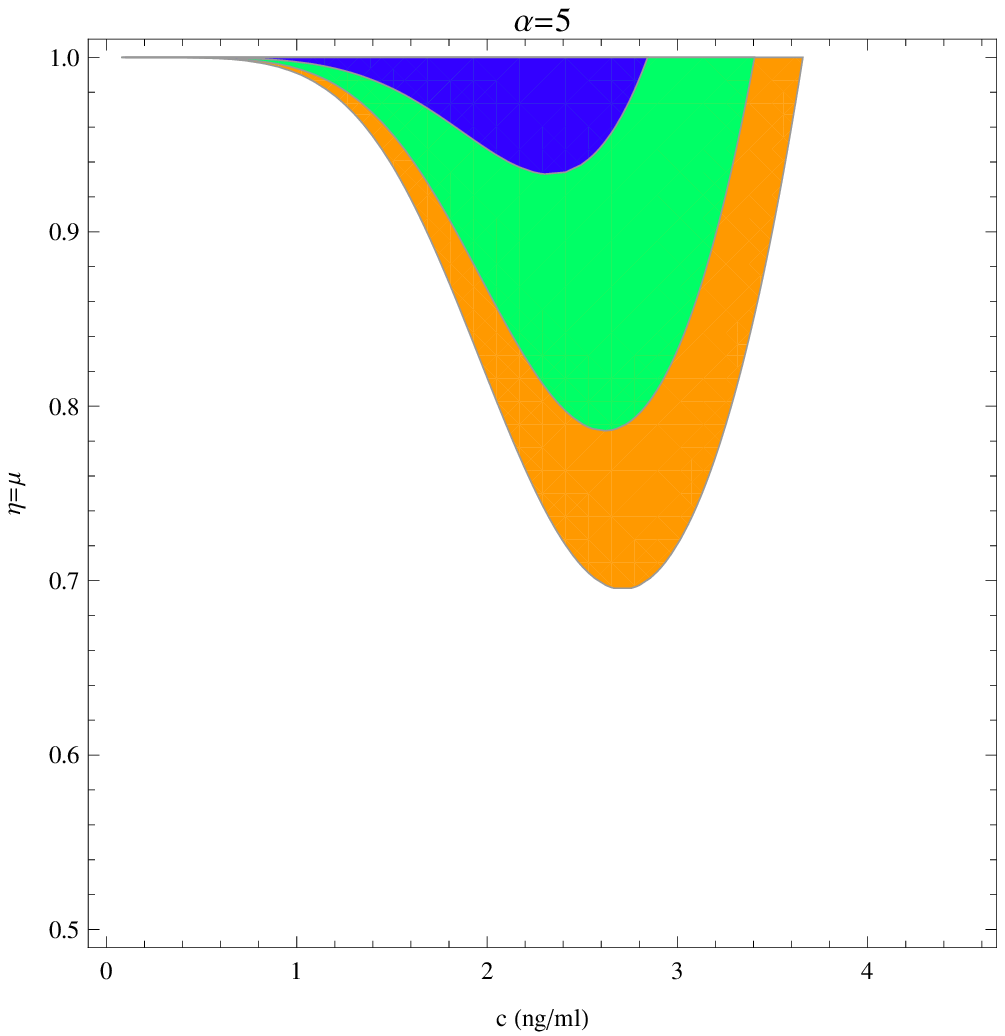} \\
\includegraphics[width=0.47\textwidth]{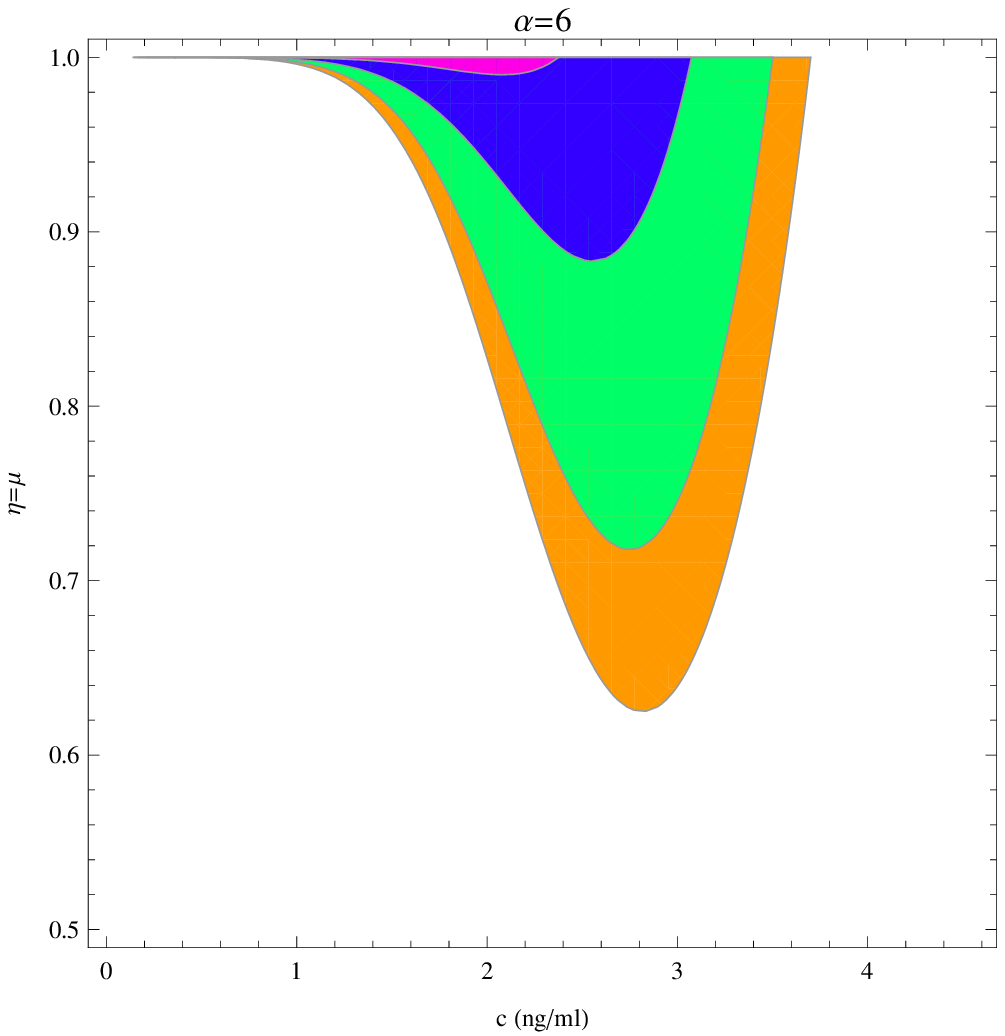} &
\includegraphics[width=0.47\textwidth]{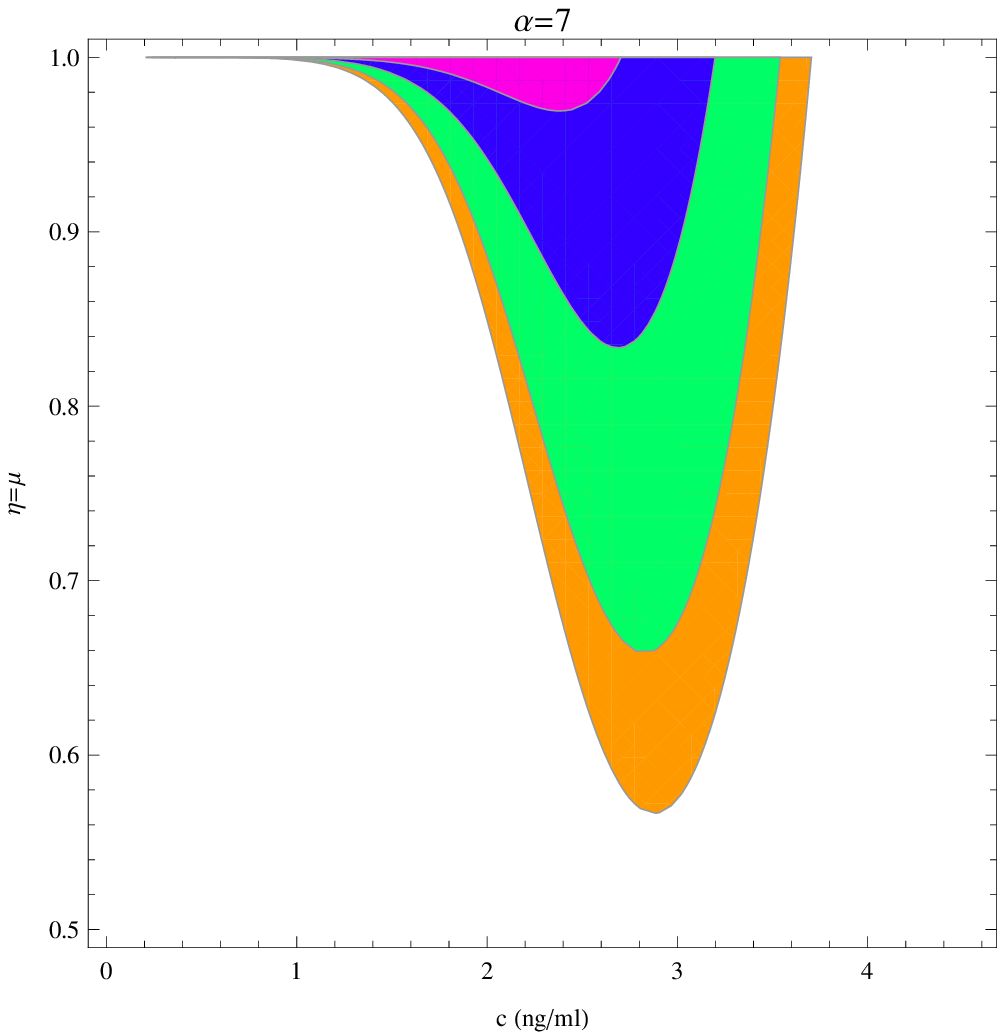}
\end{tabular}
\caption{Regions in the parameter plane $(c,\eta=\mu)$ resulting in critical values $\tau_g\in(60,90]$ (orange), $\tau_g\in(30, 60]$ (green), $\tau_g\in(15,30]$ (blue) and $\tau_g\in(0,15]$ (pink) respectively, while inequality $(\overline{I_2})$ is fulfilled, for different values of $\alpha$. (case of strong Gamma kernels)}
\label{fig:gamma}
\end{figure}

\begin{figure}
\centering
\begin{tabular}{ccc}
\includegraphics[width=0.47\textwidth]{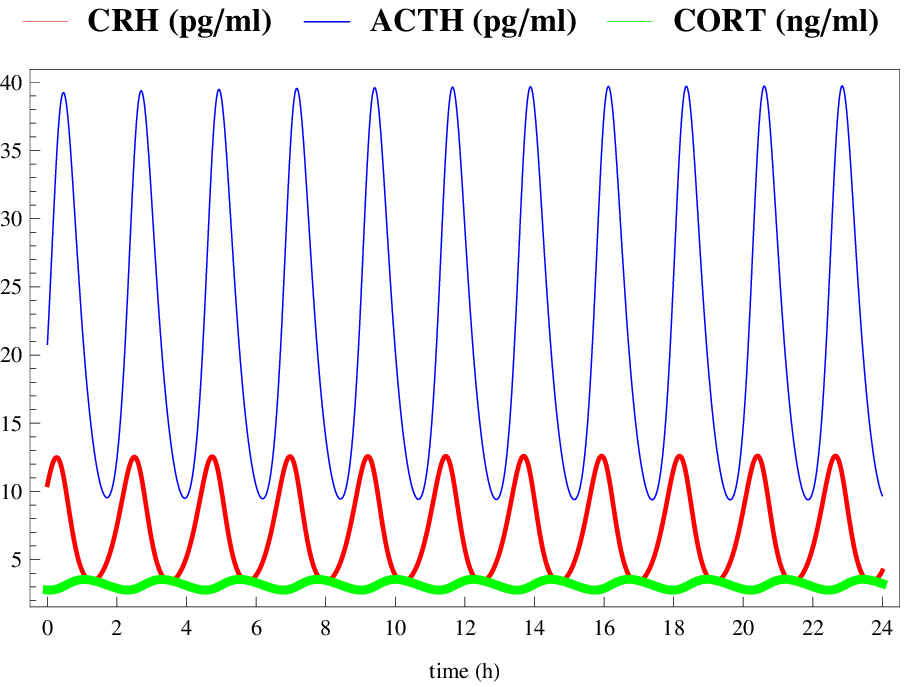} &
\includegraphics[width=0.47\textwidth]{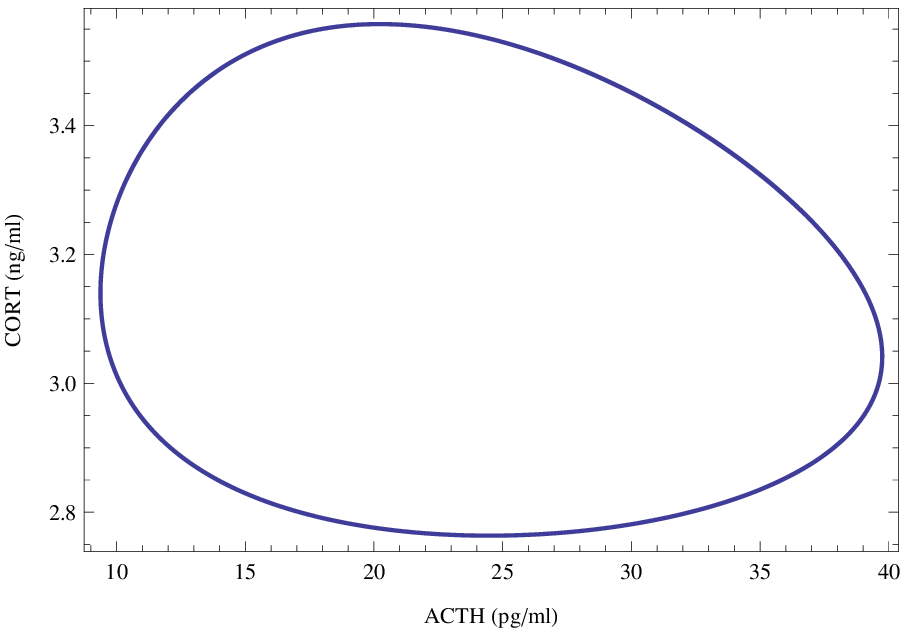} \\
\end{tabular}
\caption{Stable periodic orbit of system (\ref{sys.hpa.dd}) with strong Gamma kernels ($n_1=0$, $n_2=n_{31}=n_{32}=2$, $\beta=3.5$ (min), mean delays: $\tau_2=\tau_{31}=\tau_{32}=7$ (min)) due to the Hopf bifurcation taking place when the bifurcation parameter $\beta$ exceeds the critical value $\beta_4=3.084$ (min), in the case of feedback functions $f_1(u)=18.18\left(1-\frac{u^{6}}{2000^{6}+u^{6}}\right)$, $f_2(u)=1.3\left(1-\frac{u^{6}}{2000^{6}+u^{6}}\right)$.}
\label{fig:gamma2}
\end{figure}

\begin{figure}[htbp]
\centering
\begin{tabular}{ccc}
\includegraphics[width=0.47\textwidth]{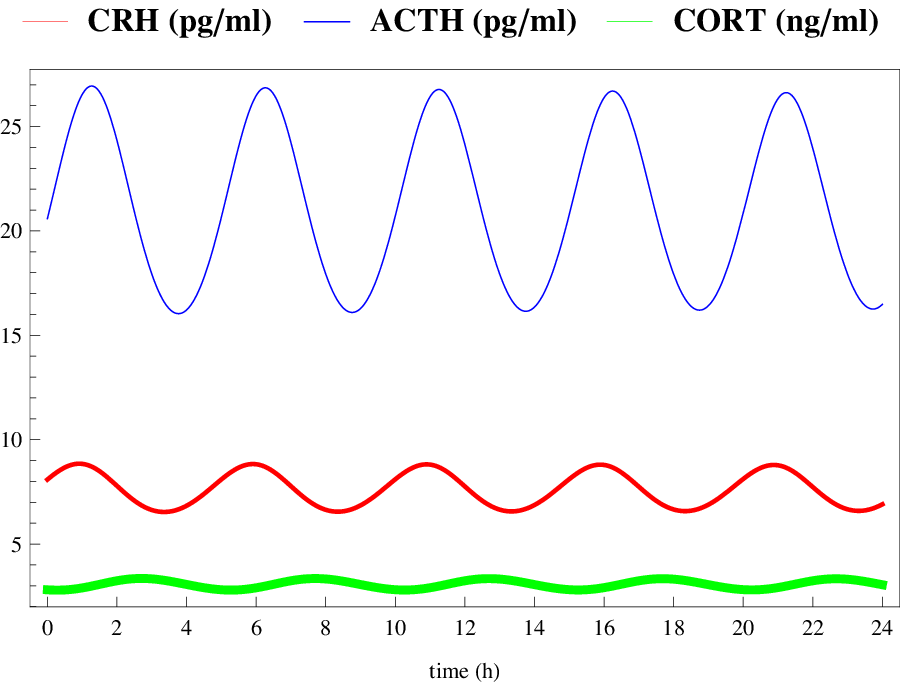} &
\includegraphics[width=0.47\textwidth]{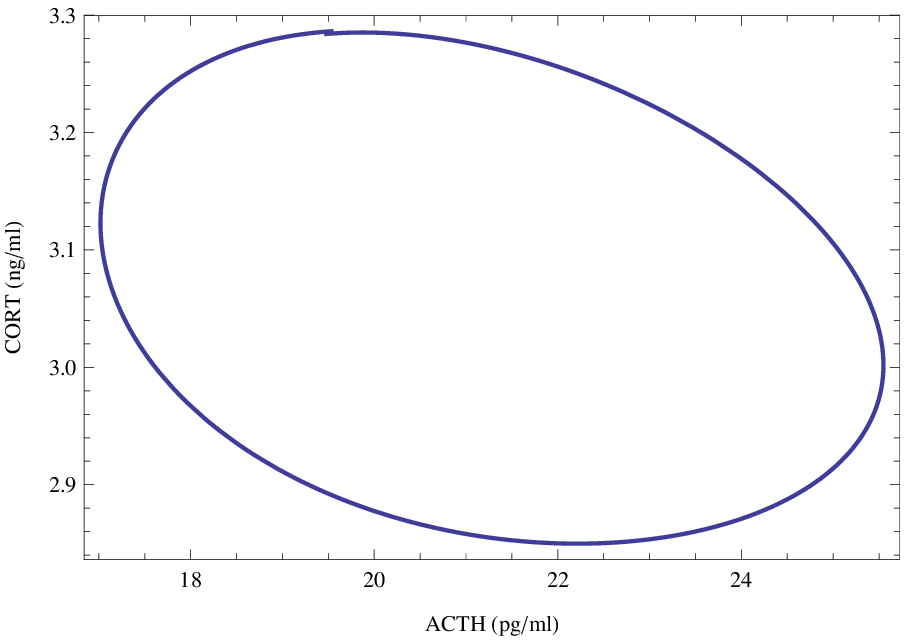} \\
\end{tabular}
\caption{Stable periodic orbit of system (\ref{sys.hpa.dd}) with strong Gamma kernels ($n_1=0$, $n_2=n_{31}=n_{32}=2$, $\beta=17$ (min), mean delays: $\tau_2=\tau_{31}=\tau_{32}=34$ (min)) due to the Hopf bifurcation taking place when the bifurcation parameter $\beta$ exceeds the critical value $\beta_4=16.9753$ (min), in the case of feedback functions $f_1(u)=5.14\left(1-\frac{u^{3}}{2000^{3}+u^{3}}\right)$, $f_2(u)=0.36\left(1-\frac{u^{3}}{2000^{3}+u^{3}}\right)$.}
\label{fig:gamma1}
\end{figure}

\subsection{Mixed kernels}

For numerical simulations in the case of mixed kernels, we choose:
\begin{itemize}
\item $h_1(t)=\delta(t)$ - no time-delay;
\item $h_2(t)=\delta(t-\tau_2)$ - Dirac kernel;
\item $h_{31}(t)=h_{32}(t)=\df{t e^{-t/\beta}}{\beta^{2}}$ strong Gamma-kernels with $n=n_{31}=n_{32}=2$.
\end{itemize}

From Theorem \ref{thm.bif.mix} and eq. (\ref{eq.tau.mix}), we can numerically determine the critical value $\tilde{\tau}_{20}$  of the average time delay due the positive feedback of the pituitary on the adrenal glands $\tau=\tau_2$, representing the Hopf bifurcation parameter.

When $\alpha=6$, $\mu=\eta=1$ and $c=2$ (ng/ml) and $\beta=3.5$, with $k_1=18.18$ (pg/(ml$\cdot$min)) and $k_2=1.3$ (min$^{-1}$), the critical value is $\tilde{\tau}_{20}=5.042$ (min). In Fig. \ref{fig:mixed2}, periodic solutions are displayed for $\tau_2=6$ (min).

On the other hand, when $\alpha=3$, $\mu=\eta=0.95$ and $c=2$ (ng/ml), with $k_1=5.14$ (pg/(ml$\cdot$min)) and $k_2=0.36$ (min$^{-1}$), the critical value is $\tilde{\tau}_{20}=22.13$ (min). In Fig. \ref{fig:mixed1}, periodic solutions are shown for $\tau_2=23$ (min).

As in the previous two cases, oscillations corresponding to smaller critical values (fast feedback) have higher amplitudes and higher frequency (over a 24 hour range) than those corresponding to larger critical values (slow feedback).

\begin{figure}
\centering
\begin{tabular}{ccc}
\includegraphics[width=0.47\textwidth]{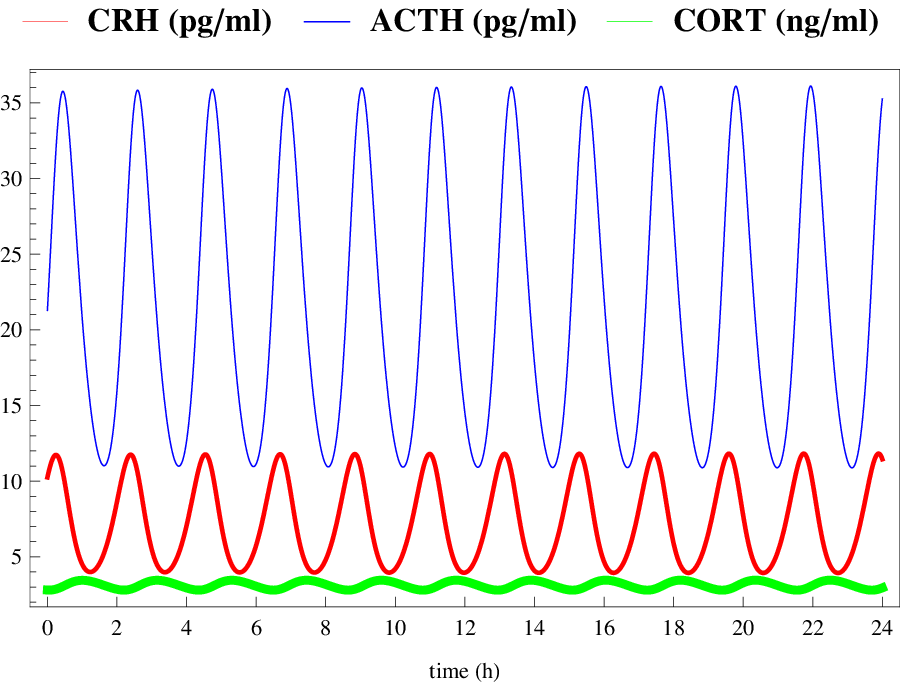} &
\includegraphics[width=0.47\textwidth]{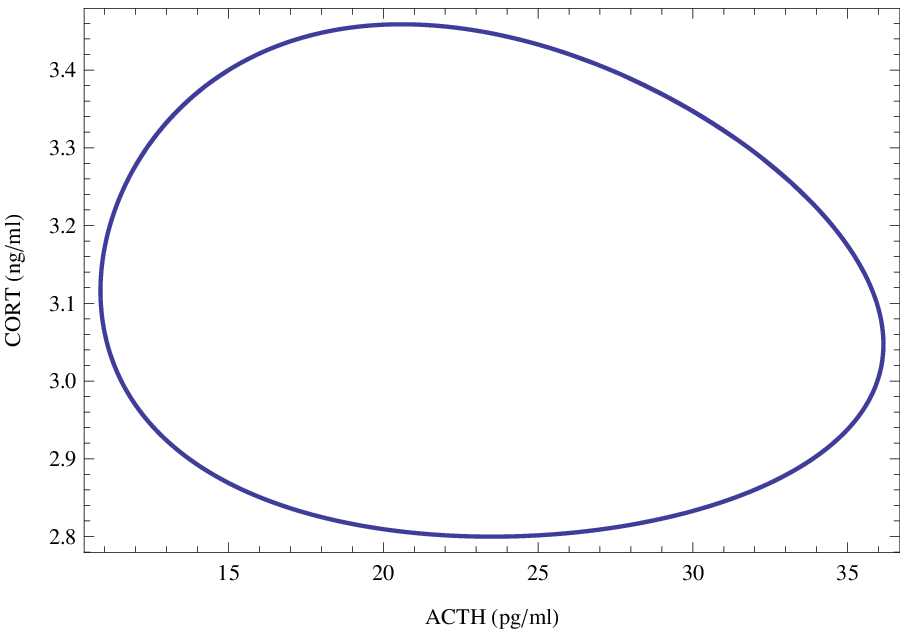} \\
\end{tabular}
\caption{Stable periodic orbit of system (\ref{sys.hpa.dd}) with mixed kernels (Dirac kernel $h_2(t)=\delta(t-\tau_2)$, with $\tau_2=6$ (min), and strong Gamma kernels $h_{31}=h_{32}$ with $n_{31}=n_{32}=2$, $\beta=3.5$ and mean delays $\tau_{31}=\tau_{32}=7$ (min)) due to the Hopf bifurcation taking place when the bifurcation parameter $\tau=\tau_2$ exceeds the critical value $\tilde{\tau}_{20}=5.042$ (min), in the case of feedback functions $f_1(u)=18.18\left(1-\frac{u^{6}}{2000^{6}+u^{6}}\right)$, $f_2(u)=1.3\left(1-\frac{u^{6}}{2000^{6}+u^{6}}\right)$.}
\label{fig:mixed2}
\end{figure}

\begin{figure}[htbp]
\centering
\begin{tabular}{ccc}
\includegraphics[width=0.47\textwidth]{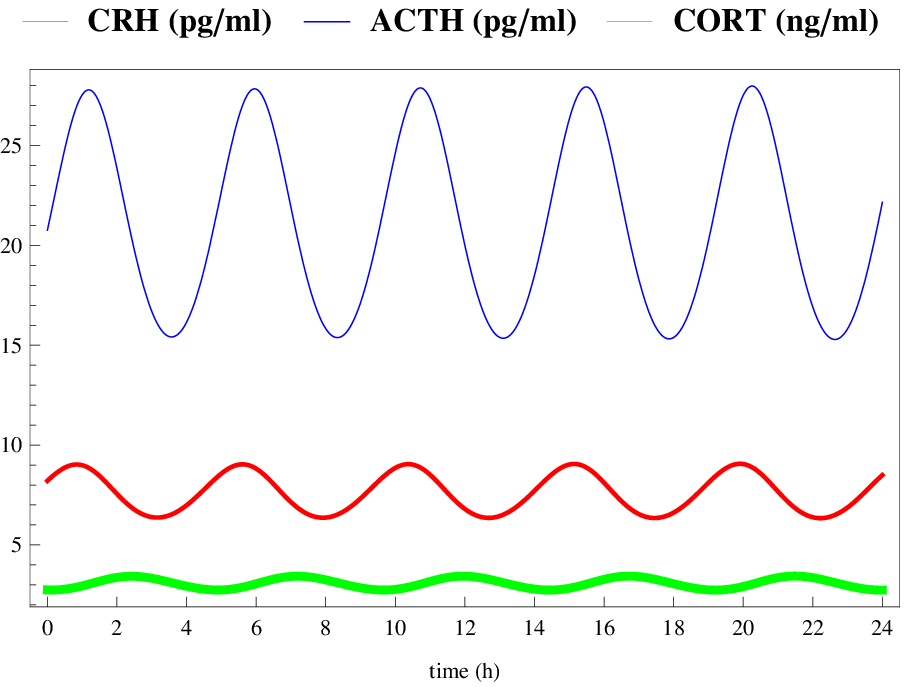} &
\includegraphics[width=0.47\textwidth]{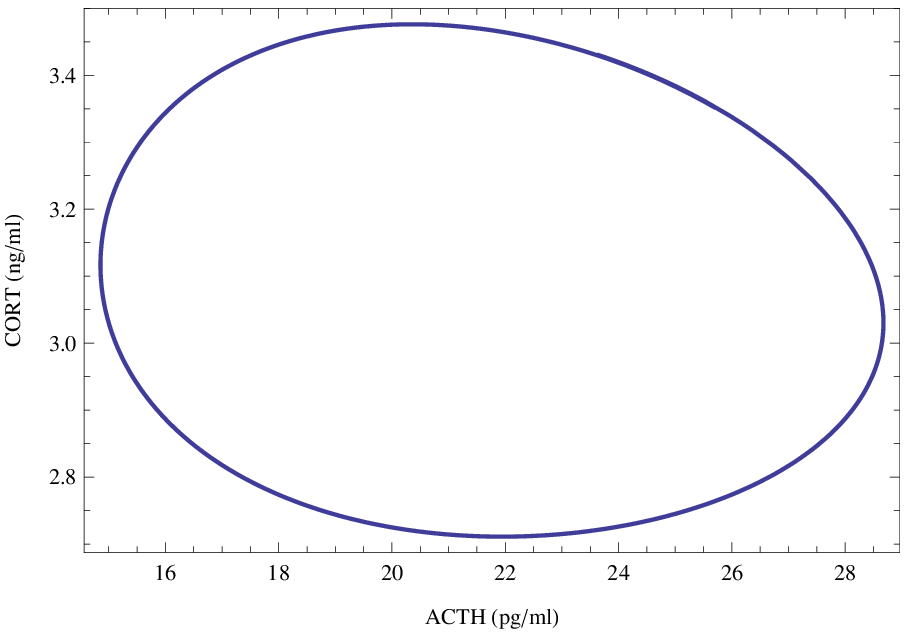} \\
\end{tabular}
\caption{Stable periodic orbit of system (\ref{sys.hpa.dd}) with mixed kernels (Dirac kernel $h_2(t)=\delta(t-\tau_2)$, with $\tau_2=23$ (min), and strong Gamma kernels $h_{31}=h_{32}$ with $n_{31}=n_{32}=2$, $\beta=20$ and mean delays $\tau_{31}=\tau_{32}=40$ (min)) due to the Hopf bifurcation taking place when the bifurcation parameter $\tau=\tau_2$ exceeds the critical value $\tilde{\tau}_{20}=22.13$ (min), in the case of feedback functions $f_1(u)=5.14\left(1-\frac{u^{3}}{2000^{3}+u^{3}}\right)$, $f_2(u)=0.36\left(1-\frac{u^{3}}{2000^{3}+u^{3}}\right)$.}
\label{fig:mixed1}
\end{figure}

\subsection{Fractional order model}

The numerical simulations for the fractional-order system (\ref{sys.hpa.frac}) have been performed using an extension of the Adams-Bashforth-Moulton predictor-corrector method presented in \cite{Diethelm}.

The delay accounting for the positive feedback of the hypothalamus on the pituitary is $\tau_1=0$. The delays due to the positive feedback of the pituitary on the adrenal glands and to the negative feedback effect of the adrenal glands on the hypothalamus and pituitary, respectively, are chosen to be equal: $\tau_2=\tau_{31}=\tau_{32}$.

When $\alpha=6$, $\mu=\eta=1$ and $c=2$ (ng/ml), with $k_1=18.18$ (pg/(ml$\cdot$min)) and $k_2=1.3$ (min$^{-1}$), a stable limit cycle has been observed numerically for $\tau_2=\tau_{31}=\tau_{32}=14$ (min) and the fractional order $q=0.9$ (see Fig. \ref{fig:frac21}). As it can been expected, if a smaller fractional order is taken into account (e.g. $q=0.8$) for the same delays, the equilibrium point $E$ is asymptotically stable (see  Fig. \ref{fig:frac22}). Smaller fractional orders are associated with a more pronounced asymptotically stable behaviour of the system.

\begin{figure}
\centering
\begin{tabular}{ccc}
\includegraphics[width=0.47\textwidth]{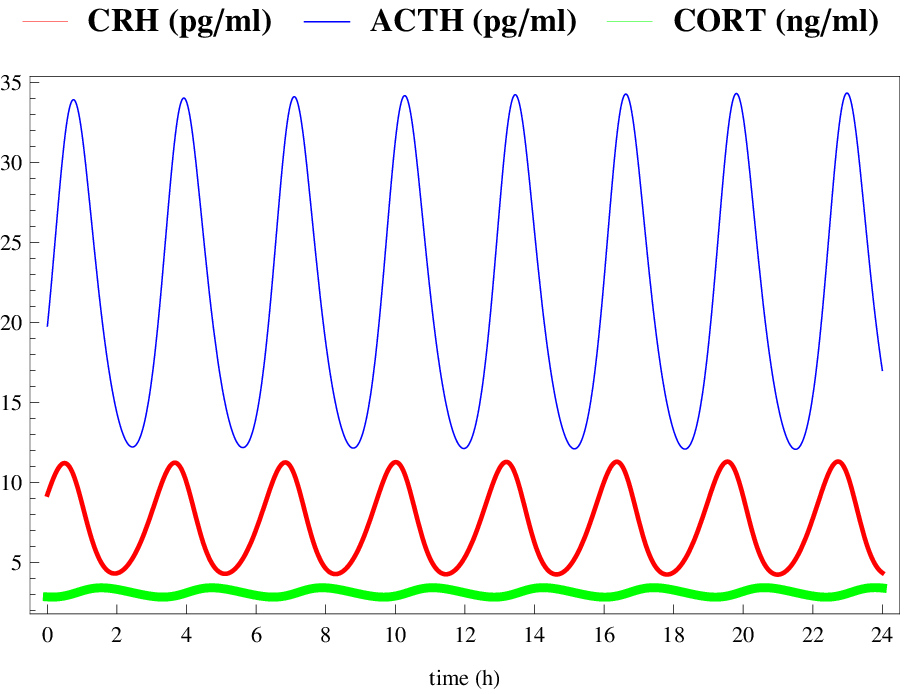} &
\includegraphics[width=0.47\textwidth]{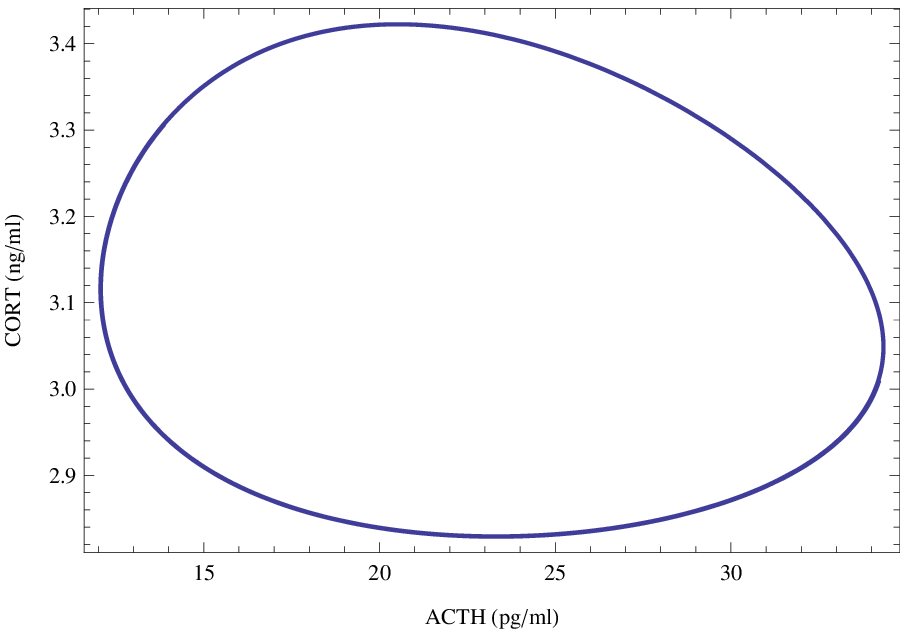} \\
\end{tabular}
\caption{Stable periodic orbit of system (\ref{sys.hpa.frac}) with fractional order $q=0.9$ and discrete delays $\tau_1=0$ and $\tau_2=\tau_{31}=\tau_{32}=14$ (min), in the case of feedback functions $f_1(u)=18.18\left(1-\frac{u^{6}}{2000^{6}+u^{6}}\right)$, $f_2(u)=1.3\left(1-\frac{u^{6}}{2000^{6}+u^{6}}\right)$.}
\label{fig:frac21}
\end{figure}

\begin{figure}
\centering
\begin{tabular}{ccc}
\includegraphics[width=0.47\textwidth]{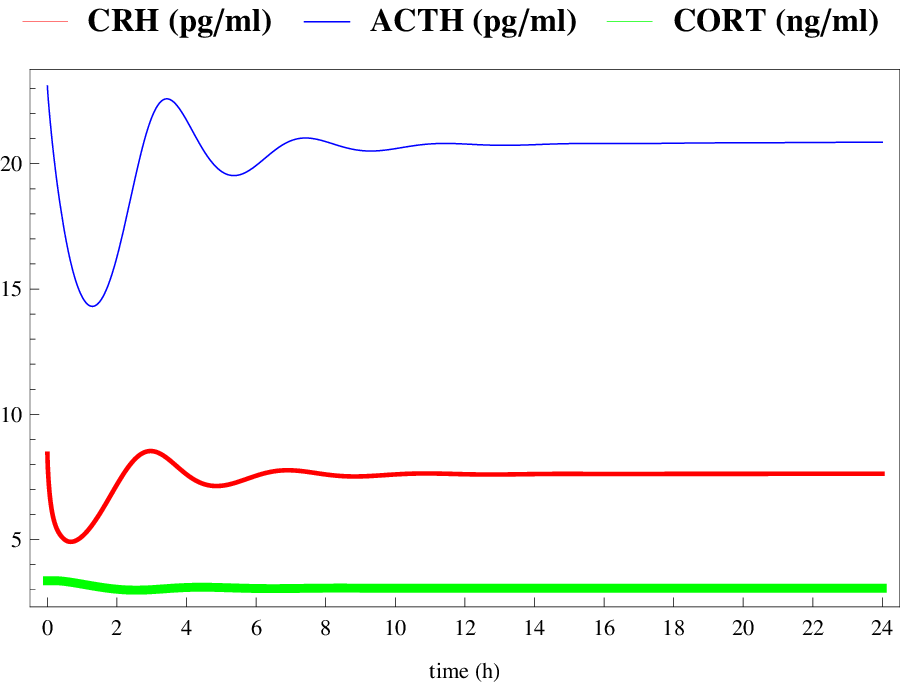} &
\includegraphics[width=0.47\textwidth]{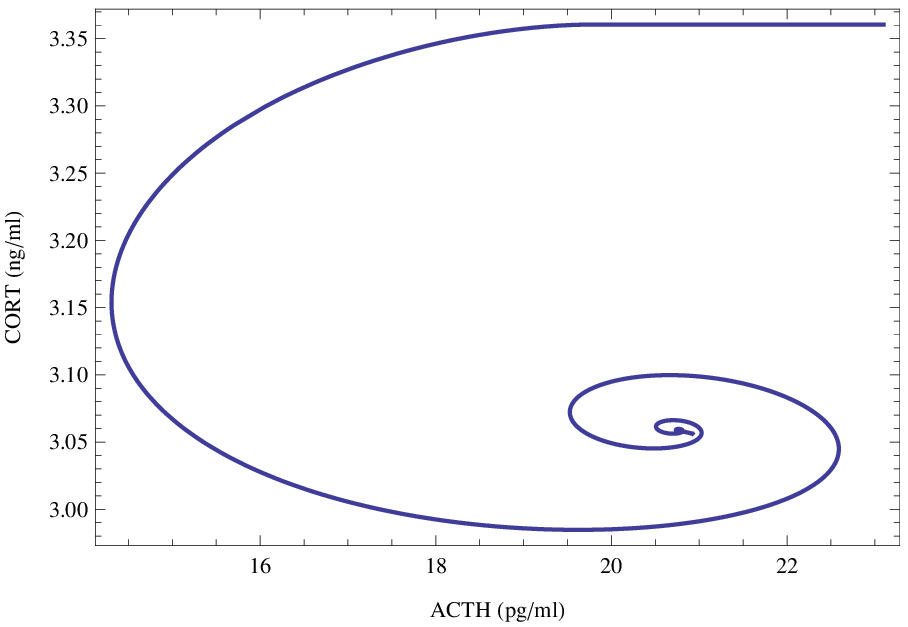} \\
\end{tabular}
\caption{Trajectories of system (\ref{sys.hpa.dd}) with fractional order $q=0.8$ and discrete delays $\tau_1=0$ and $\tau_2=\tau_{31}=\tau_{32}=14$ (min), converge to the asymptotically stable equilibrium point $E$, in the case of feedback functions $f_1(u)=18.18\left(1-\frac{u^{6}}{2000^{6}+u^{6}}\right)$, $f_2(u)=1.3\left(1-\frac{u^{6}}{2000^{6}+u^{6}}\right)$.}
\label{fig:frac22}
\end{figure}

\section{Conclusions and future work }
In this paper, we have generalized the existing minimal model of the HPA axis, firstly, by including distributed time delays and secondly, by considering frac\-tional-order derivatives. This approach to the modelling of the biological processes is more realistic because it involves memory properties, taking into account the whole past history of the variables. These models are able to capture the vital mechanisms of the HPA system.

The existence of a unique equilibrium point of the considered models has been proved. Considering general delay kernels in the model with distributed delays, delay-independent sufficient conditions for the local asymptotic stability of the unique equilibrium point have been obtained. These findings are useful if one is unable to estimate the time delays in the system.  A thorough bifurcation analysis has been undertaken in three cases: Dirac kernels, Gamma kernels, and finally, a mixed choice of Dirac and Gamma kernels. Critical values of the appropriately chosen bifurcation parameters have been found which account for the occurrence of Hopf bifurcations. { Studying the criticality of Hopf bifurcations is a laborious mathematical task, which will be addressed in a future paper. }

Extensive numerical simulations show that when the bifurcation parameters pass through the critical values, periodic solutions appear { which reproduce the ultradian rhythm of the HPA axis }. It has been observed that oscillations corresponding to smaller critical values of the bifurcation parameters (fast feedback) have higher amplitudes and higher frequency (over a 24 hour range) than those generated by larger critical values (slow feedback).

For the fractional-order mathematical model of the HPA axis, it has been shown that if no time delays are considered, the unique equilibrium point is asymptotically stable. When discrete time delays are introduced, we rely on numerical simulations to exemplify the existence of oscillatory solutions for sufficiently large subunitary values of the fractional order. Numerical simulations show that, in the presence of discrete time delays, smaller fractional orders are associated with a more pronounced asymptotically stable behaviour of the system, in a neighborhood of the equilibrium point.

{ Different approaches of the HPA axis including environmental and physiological perturbations (for example, in the form of white or colored noise, or time-varying input)  which can be modelled by stochastic  and/or  impulsive terms, will be developed as future research, with the aim of reproducing both the circadian and ultradian rhythms underlying cortisol secretion within the HPA system. }

\section*{Acknowledgements}

The authors are especially grateful to the editor and the referees for helpful comments and suggestions.

This work was supported by grants of the Romanian National Authority for Scientific Research and Innovation, CNCS-UEFISCDI, project no.
PN-II-ID-PCE-2011-3-0198 and project no. PN-II-RU-TE-2014-4-0270.

\vspace*{6pt}




\end{document}